\documentclass[11pt]{amsart}
\usepackage{amsmath,amsfonts,amssymb,amsthm,epsfig}

\usepackage[all]{xy}
\usepackage{tikz}
\usepackage{color}

\usepackage[colorlinks=true, pdfstartview=FitV, linkcolor=blue, citecolor=blue, urlcolor=blue]{hyperref}
\newcommand{\arxiv}[1]{\href{http://arxiv.org/abs/#1}{\tt arXiv:\nolinkurl{#1}}}

\newtheorem{Theorem}{Theorem}[section]
\newtheorem{Proposition}[Theorem]{Proposition} 
\newtheorem{Lemma}[Theorem]{Lemma}

\newtheorem{Corollary}[Theorem]{Corollary}

\theoremstyle{definition}
\newtheorem{Definition}[Theorem]{Definition}
\newtheorem{Remark}[Theorem]{Remark}

\newcommand{\MV}{\mathcal MV}
\newcommand{\eps}{{\varepsilon}}

\newcommand{\oa}{{\overline a}}
\newcommand{\ob}{{\overline b}}
\newcommand{\oc}{{\overline c}}
\newcommand{\od}{{\overline d}}

\newcommand{\olambda}{\overline \lambda}
\newcommand{\omu}{\overline \mu}

\newcommand{\wt}{\operatorname{wt}}

\newcommand{\bz}{\mathbb{Z}}

\newcommand{\br}{\mathbb{R}}
\newcommand{\bn}{\mathbb{N}}
\newcommand{\g}{\mathfrak{g}}

\newcommand{\asl}{\widehat{\mathrm{sl}}}

\newcommand{\Id}{\text{Id}}

\theoremstyle{definition}

\begin{document} 


\title[$\asl_2$ MV polytopes]{Rank 2 affine MV polytopes}

\author{Pierre Baumann}
\address
{Pierre Baumann
Institut de Recherche Math\'ematique Avanc\'ee,
Universit\'e de Strasbourg et CNRS,
7 rue Ren\'e Descartes,
67084 Strasbourg Cedex,
France.
}
\email{p.baumann@unistra.fr}

\author{Thomas Dunlap}
\address
{
Thomas Dunlap,
Einstein Institute of Mathematics, 
The Hebrew University of Jerusalem,
Jerusalem, 91904, Israel.
}
\email{tdunlap@umich.edu }

\author{Joel Kamnitzer}
\address{
Joel Kamnitzer,
Dept. of Mathematics,
U. of Toronto,
Toronto, ON, M5S 2E4
Canada.}
\email{jkamnitz@math.toronto.edu}

\author{Peter Tingley}
\address{Peter Tingley,
MIT dept. of math,
77 Massachusetts Ave,
Cambridge, MA, USA 02139.}
\email{ptingley@math.mit.edu}

\thanks{P.B. acknowledges support from the ANR, project ANR-09-JCJC-0102-01. T.D. acknowledges support from the ERC, project \#247049(GLC). J.K. acknowledges
support from NSERC.  P.T. acknowledges support from the NSF postdoctoral fellowship DMS-0902649.}

\begin{abstract}

We give a realization of the crystal $B(-\infty)$ for $\asl_2$ using decorated polygons. The construction and proof are combinatorial, making use of Kashiwara and Saito's characterization of $B(-\infty)$ in terms of the $*$ involution. The polygons we use have combinatorial properties suggesting they are the $\asl_2$ analogues of the Mirkovi\'c-Vilonen polytopes defined by Anderson and the third author in finite type. 
Using Kashiwara's similarity of crystals we also give MV polytopes for $A_2^{(2)}$, the only other rank two affine Kac-Moody algebra.  
\end{abstract}

\maketitle

\setcounter{tocdepth}{1} 
\tableofcontents
\section{Introduction}

Fix a symmetrizable Kac-Moody algebra $\g$. Kashiwara's theory of crystals constructs combinatorial data (a set  $B(-\lambda)$ along with some operations) corresponding to each integrable lowest weight representation $V(-\lambda)$, which records certain leading order behavior of the representation. The crystals $B(-\lambda)$ form a directed system whose limit $B(-\infty)$ can be thought of as the crystal for $U^+(\g)$. This theory makes heavy use of the quantized universal enveloping algebra $U_q(\g)$ associated with $\g$, but there are also combinatorial characterizations of $B(-\infty)$, and it can often be realized by more elementary means. 

When $\g$ is of finite type, there is a realization of $B(-\infty)$ using Mirkovi\'c-Vilonen (MV) polytopes, developed by Anderson \cite{Anderson:2003} and the third author \cite{Kamnitzer:2006, Kamnitzer:2010}. In this paper, we realize $B(-\infty)$ for $\asl_2$ using a collection of decorated polytopes in the $\asl_2$ root lattice.
Our construction has many properties in common with finite type, which we take as evidence that we have found the correct generalization of MV polytopes in type $\asl_2$: 
\begin{enumerate}
\item An MV polytope is uniquely determined by either its right or left side, and any candidate for one side does correspond to an MV polytope.  
\item MV polytopes are defined using systems of non-intersecting diagonals (Definition \ref{def:characterization}).  This is the generalization of a known characterization for types $A_2$ and $B_2$ (see \cite[Section 3.4]{Kamnitzer:2010}). 
\item Kashiwara's involution on $B(-\infty)$ is negation. 
\item The crystal operator $e_1$ (respectively $e_0$) increases the length of the top edge on the left (resp. right) side by 1, and otherwise does not affect that side. This implies that $\varphi_i$ and $ \varphi_i^* $ are given by the lengths of the top and bottom edges.

\item For a dominant weight $ \Lambda $, the lowest weight crystal $B(-\Lambda) \subset B(-\infty)$ can be characterized as the set of MV polytopes such that, if the bottom vertex is placed at $-\Lambda$, the whole polytope is contained in the convex hull of the Weyl group orbit of $-\Lambda$.
\end{enumerate}

Two recent works motivated and informed this paper. The first is the construction by the second author \cite{Dunlap:2010} of a conjectural realization of $B(-\infty)$ for $\asl_2$ using decorated lattice polytopes. Here we use new definitions for both the set of polytopes and the crystal operators, but our setup was motivated by this earlier work, and we expect they are equivalent.

The second is the construction by the other three authors \cite{BKT:2011} of MV polytopes (for symmetric finite and affine Lie algebras) associated to components of Lusztig's nilpotent varieties.  We believe that the polytopes defined here essentially agree with the $\asl_2$ case of that construction, and we plan to address this issue in future work. As discussed in \cite{BKT:2011}, this would give a combinatorial description of MV polytopes in all symmetric affine cases.

Finally, we would like to mention the  
work of Naito, Sagaki and Saito \cite{NSS:2010} (see also Muthiah \cite{Muthiah:??}) giving a version of MV polytopes for $\asl_n$, and in particular $\asl_2$. That construction seems to be quite different from the one given here, and it would be interesting to understand the relationship between them. 

\subsection{Acknowledgments}
We would like to thank Dinakar Muthiah for helpful conversations.

\section{Background}

\subsection{Crystals}

Fix a symmetrizable Kac-Moody algebra (which for most of this paper will be $\asl_2$) and let $\Gamma = (I, E)$ be its Dynkin diagram.
We are interested in the crystals $B(-\lambda)$ and $B(-\infty)$ associated with the lowest weight representation $V(-\lambda)$ and with $U^+(\g)$ respectively. These are combinatorial objects arising from the theory of crystal bases for the corresponding quantum group (see e.g. \cite{Kashiwara:1995}). This section contains a brief explanation of the results we need, roughly following \cite{Kashiwara:1995} and \cite{Hong&Kang:2000}, to which we refer the reader for details. We begin with a combinatorial notion of crystal that includes many examples which do not arise from representations. 

\begin{Definition} (see \cite[Section 7.2]{Kashiwara:1995}) A {\it combinatorial crystal} is the data $(B, e_i, f_i, \varepsilon_i, \varphi_i, \wt)$ of a set $B$ along with functions $\wt \colon B \to P$ (where $P$ is the weight
  lattice), and, for each $i \in I$, $\varepsilon_i, \varphi_i \colon B \to \bz \cup \{-\infty\}$ and $e_i, f_i: B \rightarrow B \sqcup \{ \emptyset \}$, such that
  \begin{enumerate}
  \item $\varphi_i(b) = \varepsilon_i(b) + \langle \wt(b), \alpha_i^\vee \rangle$.
  \item $e_i$ increases $\varphi_i$ by 1, decreases $\eps_i$ by 1 and increases $\wt$ by $\alpha_i$.
  \item $f_i b = b'$ if and only if $e_i b' = b$.
  \item If $\varphi_i(b) = -\infty$, then $e_i b = f_i b = \emptyset$.
  \end{enumerate}
We often denote a combinatorial crystal by its underlying set $B$, suppressing the other data.
\end{Definition}

\begin{Definition}
A morphism 
of combinatorial crystals 
is a map $\phi: B \rightarrow C$ of sets which commutes with all the structure (this is called a strict morphism in e.g. \cite{KS:1997}).   
\end{Definition}

\begin{Definition}
Let $B$ and $C$ be combinatorial crystals. The tensor product $B \otimes C$ is the cartesian product $B \times C$ with crystal operators $e_i, f_i$ defined by
\begin{equation*}
f_i (b \otimes c) = 
\begin{cases}
f_i(b) \otimes c \quad \text{ if $\eps_i(b) \geq \varphi_i(c)$} \\
b \otimes f_i(c) \quad \text{ if $\eps_i(b) < \varphi_i(c)$} 
\end{cases}
\quad \text{and} \quad
e_i (b \otimes c) = 
\begin{cases}
e_i(b) \otimes c \quad \text{ if $\eps_i(b) > \varphi_i(c)$} \\
b \otimes e_i(c) \quad \text{ if $\eps_i(b) \leq \varphi_i(c).$} 
\end{cases}
\end{equation*} 
The rest of the data is given by
$\wt(b \otimes c) = \wt(b) + \wt(c)$,
$\varphi_i(b \otimes c)= \max \{ \varphi_i(b), \varphi_i(c) + \langle \wt(b), \alpha_i^\vee \rangle \} $, and 
$\eps_i(b \otimes c)= \max \{ \eps_i(c), \eps_i(b) - \langle  \wt(c), \alpha_i^\vee \rangle \}.$
\end{Definition}

\begin{Definition}
\label{def:lwcrystal}
A {\it lowest weight} combinatorial crystal is a combinatorial crystal 
which has a distinguished element $b_-$ (the lowest weight element) such that 
\begin{enumerate}
\item The lowest weight element $b_-$ can be reached from any $b \in B$ by applying a sequence of $f_i$  for various $i \in I$. 
\item For all $b \in B$ and all $i \in I$, $\varphi_i(b) = \max \{ n : f_i^n(b) \neq \emptyset \}$.
\end{enumerate}
  \end{Definition}

\begin{Definition}[see {\cite[Section 7.5]{Kashiwara:1995}}]
Let $B^{(i)}$ be the crystal
\[
\cdots \xrightarrow{i} b^{(i)}(1) \xrightarrow{i} b^{(i)}(0) \xrightarrow{i}
b^{(i)}(-1) \xrightarrow{i} b^{(i)}(-2) \xrightarrow{i} \cdots 
\]
where $\wt(b^{(i)}(k))=k\alpha_i$, $\varphi_i(b^{(i)}(k)) = k$, $\varepsilon_i(b^{(i)}(k))=-k$,
and for $j \neq i$, $\varphi_j(b^{(i)}(k))= \eps_j(b^{(i)}(k))=-\infty.$
Here the arrows show the action of $f_i$.
\end{Definition}

The following is a specialization of a result of Kashiwara and Saito \cite[Proposition 3.2.3]{KS:1997}. 

\begin{Theorem} \label{th:comb-characterizaton}
Let $B$ be a lowest weight combinatorial crystal. Fix an involution $*$ on $B$, and define $f_i^*=*f_i*$ and $\varphi_i^*(b)= \varphi_i(*b)$. Define $\Phi_i: B \rightarrow B \otimes B^{(i)}$ by 
$$\Phi_i(b) = (f_i^*)^{\varphi_i^*(b)}(b) \otimes b^{(i)}(\varphi_i^*(
    b)).$$
If $\Phi_i$ is a morphism of crystals for all $i$ then $B\simeq B(-\infty)$.
\qed
\end{Theorem}

It is shown in \cite[Theorem 2.2.1]{Kashiwara:1993} that there is an involution $*$ ({\it Kashiwara's involution}) on $B(-\infty)$ such that the conditions of Theorem \ref{th:comb-characterizaton} hold for $(B(-\infty), *)$. Furthermore $*$ is uniquely characterized by the conditions of Theorem \ref{th:comb-characterizaton}, as these conditions uniquely determine the operators $f_i^*$ and hence the operators $e_i^*$, and $*$ is determined by $*(e_{i_K} \cdots  e_{i_1} v_-)= e^*_{i_K} \cdots  e^*_{i_1} v_-$ for all $i_1, \ldots i_K \in I.$ The involution $*$ also has a simple algebraic interpretation, which can be found in \cite[Section 8.3]{Kashiwara:1995} (see also \cite[Theorem 14.4.3]{Lusztig:1993}).

It will actually be convenient for us to use the following dual version of Theorem \ref{th:comb-characterizaton}, where the roles of the unstarred and starred crystal operators are reversed. 

\begin{Corollary} \label{cor:comb-characterizaton2}
Let $B$ be a lowest weight combinatorial crystal. Fix an involution $*$ on $B$. We abuse notation and also denote the involution $* \otimes \Id$ of $B \otimes B^{(i)}$ by $*$. 
Define $\Phi_i: B \rightarrow B \otimes B^{(i)}$ by
$$\Phi_i(b) = (f_i)^{\varphi_i(b)}(b) \otimes b^{(i)}(\varphi_i(
    b)).$$
If $\Phi_i$ commutes with all the operators $f_i^*:= *  f_i  *$, then $B\simeq B(-\infty)$. Equivalently, if the following three conditions hold, then $B\simeq B(-\infty)$:
\begin{enumerate}
\item If $i \neq j$ then, for all $b \in B$, $f_i^*f_j(b)= f_jf_i^*( b)$, 

\item If $\eps_i^*( (f_i)^{\varphi_i(b)}(b)) < \varphi_i( b)$ then $f_i(b) = f_i^*(b)$, and

\item If $\eps_i^*( (f_i)^{\varphi_i(b)}(b)) \geq \varphi_i( b)$ then  $f_i^k f_i^*(b) = f_i^*f_i^k(b)$ for all $k \geq 0$. 

\end{enumerate}
\end{Corollary}

\begin{proof}
The first part follows from Theorem \ref{th:comb-characterizaton} simply by twisting by $*$. The second statement just gives the conditions one needs to explicitly check to see that the maps $\Phi_i$ all commute with the $*$ crystal operators.
\end{proof}

\begin{Theorem} \cite[Proposition 8.2]{Kashiwara:1995} \label{th:sub-crystal} Fix a dominant integral weight $\lambda = \sum_{i \in I} a_i \omega_i$, where the $\omega_i$ are the fundamental weights. 
Let $B^\lambda$ be the subgraph of $B(-\infty)$ consisting of those vertices $b$ such that $\varphi_i^*(b) \leq a_i$. For each $i \in I$, $f_i(B^\lambda) \subset B^\lambda \sqcup \{ \emptyset \}$, so one can consider the restriction $f^\lambda_i$ of $f_i$ to $B^\lambda$. Let $e^\lambda_i: B^\lambda \rightarrow B^\lambda \sqcup \{ \emptyset \}$ be the operator obtained from $e_i$ by setting $e^\lambda_i(b)= \emptyset$ if $e_i(b) \not \in B^\lambda$. Let $\wt^\lambda = \wt- \lambda$, $\eps_i^\lambda=\eps _i+ a_i$, and $\varphi^\lambda_i= \varphi_i$. Then $(B^\lambda, e_i^\lambda, f_i^\lambda, \eps^\lambda_i, \varphi^\lambda_i, \wt^\lambda)$ is isomorphic to $B(-\lambda)$.
\qed
\end{Theorem}

Finally, we will need the following result of Kashiwara when we discuss crystals of type $ A_2^{(2)} $.  Our statement is about $B(-\infty)$ as opposed to $B(-\lambda)$, but it follows immediately from Kashiwara's result by taking a direct limit.

\begin{Theorem} (see \cite[Theorem 5.1]{Kashiwara:1996}) \label{th:similarity}
Fix symmetrized Cartan matrices $N$ and $N'$, both indexed by $I$. Assume $M = \text{diag} \{ m_i \}_{i \in I}$ is a diagonal matrix such that $N' = M N M$. Then there is a unique embedding $S: B^{N'}(-\infty) \rightarrow B^N(-\infty)$ such that
\begin{enumerate}
\item $S(b^{N'}_-) = b^N_-$, where $b^{N}_-$ and $b^{N'}_-$ are the lowest weight elements of $B^N(-\infty)$ and $B^{N'}(-\infty)$ respectively. 

\item For all $b \in B^{N'}(-\infty)$ and each $i \in I$, $S(e_i(b))= e_i^{m_i}(S(b))$. \qed
\end{enumerate}
\end{Theorem}

\subsection{The $\asl_2$ root system}
We refer the reader to e.g. \cite{Kac:1990} for details. The $\asl_2$ root system $\Delta$  is the affine root system corresponding to the affine Dynkin diagram 
$$
\begin{aligned}
\bullet & \Longleftrightarrow \bullet \\
0 & \quad \quad \;\; 1 \;\; .
\end{aligned}
$$
The Cartan matrix for the corresponding Kac-Moody algebra is
$$N = 
\left(
\begin{array}{rr}
2 & -2 \\
-2 & 2
\end{array}
\right).
$$
We denote the simple roots by $\alpha_0, \alpha_1$. Define $\delta = \alpha_0+ \alpha_1$. 

Recall that the $\asl_2$ weight space is a three dimensional vector space containing $\alpha_0, \alpha_1$. This has a standard non-degenerate inner product $(\cdot, \cdot)$ such that 
$$(\alpha_i, \alpha_j) =
\begin{cases}
\;\;2 \quad \text{ if $i=j$} \\
-2 \quad \text{ if } i \neq j.
\end{cases}
$$
Notice that $(\alpha_0, \delta)= (\alpha_1, \delta)=0$. Fix  fundamental coweights $\omega_0, \omega_1$ which satisfy $( \alpha_i, \omega_j) = \delta_{i,j}$. Since this case is symmetric, these can also be taken to be the fundamental weights under the identification of weight space with coweight space. 

The set of positive roots is
\begin{equation} \label{eq:roots}
\{\alpha_0, \alpha_0+\delta, \alpha_0+2\delta, \ldots \} \sqcup \{\alpha_1, \alpha_1+\delta, \alpha_1+2\delta, \ldots \} \sqcup \{\delta, 2\delta, 3\delta \ldots \},
\end{equation}
where the first two families consist of real roots and the third family consists of imaginary roots. All imaginary roots have multiplicity 1.  We draw these roots in the plane as
\begin{center}
\begin{tikzpicture}[scale=0.3]

\draw[line width = 0.05cm, ->] (0,0)--(3,1);
\draw[line width = 0.05cm, ->] (0,0)--(-3,1);

\draw[line width = 0.05cm, ->] (0,0)--(3,3);
\draw[line width = 0.05cm, ->] (0,0)--(-3,3);

\draw[line width = 0.05cm, ->] (0,0)--(3,5);
\draw[line width = 0.05cm, ->] (0,0)--(-3,5);

\draw[line width = 0.05cm, ->] (0,0)--(3,7);
\draw[line width = 0.05cm, ->] (0,0)--(-3, 7);

\draw[line width = 0.05cm, ->] (0,0)--(0,2);
\draw[line width = 0.05cm, ->] (0,0)--(0,4);
\draw[line width = 0.05cm, ->] (0,0)--(0,6);

\draw node at (0,8) {.};
\draw node at (0,7.6) {.};
\draw node at (0,7.2) {.};

\draw node at (-3,8.8) {.};
\draw node at (-3,8.4) {.};
\draw node at (-3,8) {.};

\draw node at (3,8.8) {.};
\draw node at (3,8.4) {.};
\draw node at (3,8) {.};

\draw node at (-4.9,8.8) {.};
\draw node at (-4.9,8.4) {.};
\draw node at (-4.9,8) {.};

\draw node at (4.9,8.8) {.};
\draw node at (4.9,8.4) {.};
\draw node at (4.9,8) {.};

\draw node at (-4,1) {{\small $ \alpha_0$}};
\draw node at (-4.8,3) {{\small $\alpha_0+\delta$}};
\draw node at (-5.1,5) {{\small $\alpha_0+2\delta$}};
\draw node at (-5.1,7) {{\small $\alpha_0+3\delta$}};

\draw node at (7,0) {{.}};
\draw node at (4,1) {{\small $\alpha_1$}};
\draw node at (4.8,3) {{\small $\alpha_1+\delta$}};
\draw node at (5.1,5) {{\small $\alpha_1+2\delta$}};
\draw node at (5.1,7) {{\small $\alpha_1+3\delta$}};

\draw node at (0,9) {{\small $k \delta$}};

\end{tikzpicture}
\end{center}

\section{$\asl_2$ MV polytopes} \label{sec:GGMS}

\begin{Definition}
An $\asl_2$ GGMS polytope is a convex polytope in $\text{span}_\br \{ \alpha_0, \alpha_1 \}$ such that all edges are parallel  to roots (see \eqref{eq:roots}). Such a polytope is called integral if all vertices lie in $\text{span}_\bz \{ \alpha_0, \alpha_1 \}$. 
\end{Definition}

We can encode a GGMS polytope by recording the position of each vertex. The vertices are labeled $\mu_k, \mu^k, \overline \mu_k, \omu^k, \mu_\infty, \mu^\infty, \omu_\infty, \omu^\infty$ as in Figure \ref{MVpoly}. For each $k \geq 1$, define $a_k, a^k, \oa_k, \oa^k$ by
\begin{equation}
\begin{aligned}
\mu_k-\mu_{k-1} = a_k (\alpha_1+ (k-1) \delta), \quad \mu^{k-1}-\mu^{k} = a^k (\alpha_0+ (k-1) \delta) \\
\omu_k-\omu_{k-1} = \oa_k (\alpha_0+ (k-1) \delta), \quad \omu^{k-1}-\omu^{k} = \oa^k (\alpha_1+ (k-1) \delta). 
\end{aligned}
\end{equation}
A polytope only has finitely many vertices, so the vertices $\mu_k$ must all coincide for sufficiently large $N$, as must the vertices $\mu^k,$ $\omu_k$, $\omu^k$. We denote 
\begin{equation}
\mu_\infty = \lim_{k \rightarrow \infty} \mu_k, \; \;\mu^\infty = \lim_{k \rightarrow \infty} \mu^k, \;\;
\omu_\infty = \lim_{k \rightarrow \infty} \omu_k, \;\;  \omu^\infty = \lim_{k \rightarrow \infty} \omu^k.
\end{equation}

\begin{Definition}
A decorated GGMS polytope is a GGMS polytope along with a choice of two sequences 
$\lambda = (\lambda_1 \geq  \lambda_2 \geq \cdots)$ and $\olambda = (\olambda_1 \geq \olambda_2 \geq \cdots)$ of non-negative real numbers such that $\mu^\infty - \mu_\infty = |\lambda| \delta$, $\omu^\infty - \omu_\infty = |\olambda| \delta$, and for all sufficiently large $N$, $\lambda_N=\olambda_N=0$. Here $|\lambda|= \lambda_1 + \lambda_2+ \cdots$ and $|\olambda|= \olambda_1 + \olambda_2 + \cdots$. 
A decorated GGMS polytopes is called integral if the underlying GGMS polytope is integral and all $\lambda_k, \olambda_k$ are integers.
\end{Definition}

\begin{Definition}
The right Lusztig data of a decorated GGMS polytope is the data ${\bf a}= ( a_k, \lambda_k, a^k )_{k \in \bn}$. The left Lusztig data is
 ${\bf \oa}=( \oa_k, \olambda_k, \oa^k )_{k \in \bn}$.
\end{Definition}

\begin{Definition} \label{def:characterization} An $\asl_2$ MV polytope $P$ is a decorated $\asl_2$ GGMS polytope such that
\begin{enumerate}
\item \label{top-part} For each $k \geq 2$, $( \omu_{k}-\mu_{k-1}, \omega_1) \leq 0$ and $(\mu_{k} -\omu_{k-1}, \omega_0) \leq 0$, with at least one of these being an equality.

\item \label{bottom-part}  For each $k \geq 2$, $(\omu^{k} -\mu^{k-1}, \omega_0) \geq 0$ and $(\mu^{k} -\omu^{k-1}, \omega_1) \geq 0$, with at least one of these being an equality.

\item \label{part:middle1}If $(\mu_\infty, \omu_\infty)$ and $(\mu^\infty, \omu^\infty)$ are parallel then $\lambda=\olambda$. Otherwise, one is obtained from the other by removing a part of size $ (\mu_\infty - \omu_\infty, \alpha_1)/2$ (i.e. the width of the polytope). 

\item \label{part:middle2} $ \lambda_1, \olambda_1  \leq (\mu_\infty - \omu_\infty, \alpha_1)/2.$
\end{enumerate}
We denote by $\MV$ the set of integral $\asl_2$ MV polytopes up to translation (i.e. $ P_1 = P_2 $ in $\MV $ if $ P_1 = P_2 + \mu $ for some weight $\mu$). 
\end{Definition}

\begin{figure}[ht]

\begin{center}
\begin{tikzpicture}[yscale=0.2, xscale=0.6]

\draw [line width = 0.01cm, color=gray] (-1.5,2.5) -- (3.5, 7.5);
\draw [line width = 0.01cm, color=gray] (-2,4) -- (4, 10);
\draw [line width = 0.01cm, color=gray] (-2.5,5.5) -- (4, 12);
\draw [line width = 0.01cm, color=gray] (-3.3333,8.6666) -- (4, 16);
\draw [line width = 0.01cm, color=gray] (-3.6666,10.3333) -- (4, 18);
\draw [line width = 0.01cm, color=gray] (-4.25,13.75) -- (4, 22);
\draw [line width = 0.01cm, color=gray] (-4.5,15.55) -- (4, 24);
\draw [line width = 0.01cm, color=gray] (-4.75,17.25) -- (4, 26);
\draw [line width = 0.01cm, color=gray] (-5,19) -- (4, 28);
\draw [line width = 0.01cm, color=gray] (-5,21) -- (4, 30);
\draw [line width = 0.01cm, color=gray] (-5,23) -- (4, 32);
\draw [line width = 0.01cm, color=gray] (-5,25) -- (4, 34);
\draw [line width = 0.01cm, color=gray] (-4.6666,29.3333) -- (3.6666, 37.6666);
\draw [line width = 0.01cm, color=gray] (-4.3333,31.6666) -- (3.33333, 39.3333);

\draw [line width = 0.01cm, color=gray] (1,1) -- (-2,4);
\draw [line width = 0.01cm, color=gray]  (2,2)-- (-3,7);
\draw [line width = 0.01cm, color=gray] (2.5,3.5) -- (-3.5,9.5);
\draw [line width = 0.01cm, color=gray] (3.3333,6.6666) -- (-4.3333,14.3333);
\draw [line width = 0.01cm, color=gray] (3.6666,8.3333) -- (-4.6666,16.6666);
\draw [line width = 0.01cm, color=gray] (4,12) -- (-5,21);
\draw [line width = 0.01cm, color=gray] (4,14) -- (-5,23);
\draw [line width = 0.01cm, color=gray] (4,16) -- (-5,25);
\draw [line width = 0.01cm, color=gray] (4,18) -- (-5,27);

\draw [line width = 0.01cm, color=gray] (4,20) -- (-4.75,28.75);
\draw [line width = 0.01cm, color=gray] (4,22) -- (-4.5,30.5);
\draw [line width = 0.01cm, color=gray] (4,24) -- (-4.25,32.25);
\draw [line width = 0.01cm, color=gray] (4,28) -- (-3.5,35.5);

\draw [line width = 0.01cm, color=gray] (4,32) -- (-2,38);
\draw [line width = 0.01cm, color=gray] (4,34) -- (-1,39);
\draw [line width = 0.01cm, color=gray] (4,36) -- (0,40);
\draw [line width = 0.01cm, color=gray] (3.5,38.5) -- (1,41);

\draw [line width = 0.01cm, color=gray] (-3,7) -- (4, 14);
\draw [line width = 0.01cm, color=gray] (-4,12) -- (4, 20);

\draw [line width = 0.01cm, color=gray] (-3,37) -- (4, 30);
\draw [line width = 0.01cm, color=gray] (-4,34) -- (4, 26);
\draw [line width = 0.01cm, color=gray] (-5,27) -- (4, 18);

\draw (0,0) node {$\bullet$};

\draw (-1,1) node {$\bullet$};
\draw (-3,7) node {$\bullet$};
\draw (-4,12) node  {$\bullet$};
\draw (-5,19) node  {$\bullet$};
\draw (-5,23) node {$\bullet$};
\draw (-5,25) node {$\bullet$};
\draw  (-5,27) node {$\bullet$};
\draw (-4,34) node {$\bullet$};
\draw (-3,37) node {$\bullet$};
\draw  (2,42)  node {$\bullet$};

\draw (2,2)  node {$\bullet$};
\draw (3,5) node {$\bullet$};
\draw (4,10) node {$\bullet$};
\draw (4,28) node {$\bullet$};

\draw (4,28) node {$\bullet$};
\draw (4,32) node {$\bullet$};
\draw (4,34) node {$\bullet$};
\draw (4,36) node {$\bullet$};
\draw (3,41) node {$\bullet$};

\draw [line width = 0.04cm] (0,0)--(-1,1);
\draw  [line width = 0.04cm] (-1,1)--(-3,7);
\draw [line width = 0.04cm] (-3,7)--(-4,12);
\draw [line width = 0.04cm] (-4,12)--(-5,19);
\draw [line width = 0.04cm] (-5,19)--(-5,27);
\draw [line width = 0.04cm] (-5,27)--(-4,34);
\draw [line width = 0.04cm] (-4,34)--(-3,37);
\draw [line width = 0.04cm](-3,37)--(2,42);

\draw [line width = 0.04cm](0,0)--(2,2);
\draw [line width = 0.04cm] (2,2)--(3,5);
\draw [line width = 0.04cm] (3,5)--(4,10);
\draw [line width = 0.04cm] (4,10)--(4,36);
\draw [line width = 0.04cm](4,36)--(3,41);
\draw [line width = 0.04cm] (3,41)--(2,42);

\draw [line width = 0.05cm] (-1,1) -- (3,5);
\draw [line width = 0.05cm] (3,5) -- (-4,12);
\draw [line width = 0.05cm] (4,10) -- (-5,19);
\draw [line width = 0.05cm] (-5,27) -- (4,36);
\draw [line width = 0.05cm] (-4,34) -- (3,41);

\draw (1,1) node {\tiny $\bullet$};

\draw (4,12) node {\tiny $\bullet$};
\draw (4,14) node {\tiny $\bullet$};
\draw (4,16) node {\tiny $\bullet$};
\draw (4,18) node {\tiny $\bullet$};
\draw (4,20) node {\tiny $\bullet$};
\draw (4,22) node {\tiny $\bullet$};
\draw (4,24) node {\tiny $\bullet$};
\draw (4,26) node {\tiny $\bullet$};
\draw (4,30) node {\tiny $\bullet$};

\draw (-2,4) node {\tiny $\bullet$};
\draw (-5,21) node {\tiny $\bullet$};

\draw (-2,38) node {\tiny $\bullet$};
\draw (-1,39) node {\tiny $\bullet$};
\draw (0,40) node {\tiny $\bullet$};
\draw (1,41) node {\tiny $\bullet$};

\draw (1.2,-0.35) node {\tiny $\alpha_1$};
\draw (3,2.7) node {\tiny $\alpha_1+\delta$};
\draw (4.4,7) node {\tiny $\alpha_1+2\delta$};
\draw (4.6,23) node {\tiny $\delta$};
\draw (4.4,39) node {\tiny $\alpha_0+2\delta$};
\draw (2.8,42.5) node {\tiny $\alpha_0$};

\draw (-0.5,41) node {\tiny $\alpha_1$};
\draw (-4.15,36.3) node {\tiny $\alpha_1+\delta$};
\draw (-5.5,31) node {\tiny $\alpha_1+3\delta$};
\draw (-5.5,23) node {\tiny $\delta$};
\draw (-5.5,15) node {\tiny $\alpha_0+3\delta$};
\draw (-4.5,9) node {\tiny $\alpha_0+2\delta$};
\draw (-3,3.5) node {\tiny $\alpha_0+\delta$};
\draw (-0.7,-0.7) node {\tiny $\alpha_0$};

\draw(0,-1) node {$\mu_0 $};
\draw(2.3,0.3) node {$\mu_1$};

\draw(3.5,4.5) node {$\mu_2$};
\draw(7,10) node {$\mu_3= \mu_4 = \cdots = \mu_\infty$};
\draw(7,36) node {$\mu^3= \mu^4 = \cdots = \mu^\infty$};
\draw(4.3,41.7) node {$\mu^1= \mu^2$};

\draw(-1.4,0) node {$\overline \mu_1$};
\draw(-3.5,6) node {$\overline \mu_2$};
\draw(-4.5,12) node {$\overline \mu_3$};
\draw(-8,19) node {$\overline \mu_\infty = \cdots = \overline \mu_5= \overline \mu_4$};
\draw(-8,27) node {$\overline \mu^\infty = \cdots = \overline \mu^5= \overline \mu^4$};
\draw(-5.3,34) node {$\overline \mu^3= \overline \mu^2$};
\draw(-3.2,38.2) node {$\overline \mu^1$};
\draw(-5.3,34) node {$\overline \mu^3= \overline \mu^2$};
\draw(2.1,43.7) node {$\mu^0$};

\draw[line width = 0.03cm, ->] (0,-4)--(1,-3);
\draw[line width = 0.03cm, ->] (0,-4)--(-1,-3);
\draw (1.5,-2.6) node {$\alpha_1$};
\draw (-1.5,-2.6) node {$\alpha_0$};

\end{tikzpicture}

\end{center}

\caption{\label{MVpoly}An integral $\asl_2$ MV polytope. The partitions labeling the vertical edges are indicated by including extra vertices on the vertical edges, such that the edge is cut into the pieces indicated by the partition. 
Here 
$$ \hspace{-3cm} a_1=2, \; a_2=1, \; a_3 =1,\; \lambda_1 = 9, \; \lambda_2=2,\; \lambda_3=1,\; \lambda_4=1, \; a^3=1, \; a^1=1, $$
$$\hspace{-2.5cm} \oa_1 =1,\;  \oa_2 = 2, \; \oa_3=1, \; \oa_4 = 1, \; \olambda_1=2, \; \olambda_2=1, \; \olambda_3=1, \; \oa^4=1, \; \oa^2=1, \; \oa^1=5,$$
and all other $a_k, a^k, \oa_k, \oa^k, \lambda_k, \olambda_k$ are $0$. 
The bold diagonals form a complete system $S$ of active diagonals (note that $(\omu_2, \mu_1)$ is actually also active, so there are two choices of such a complete system).
By remark \ref{rem:cut-into-quads} all the quadrilaterals obtained by cutting the polytope along the diagonals in $S$ are themselves MV polytopes.
}
\end{figure}

\begin{Remark} \label{rem:rotate} 
It is immediate from Definition \ref{def:characterization} that the set $\MV$ is preserved under vertical reflection (i.e the linear map defined by $ \alpha_1 \mapsto -\alpha_0 $ and $ \alpha_0 \mapsto -\alpha_1 $), horizontal reflection (i.e the linear map $\alpha_0 \leftrightarrow \alpha_1$ ) and negation (which is just the composition of the first two maps). 
 \end{Remark}

\begin{Definition} \label{def:diagonals}
 The {\it active diagonals} of an MV polytope are the diagonals $(\omu_{k}, \mu_{k-1})$, $(\mu_{k}, \omu_{k-1})$, $(\omu^{k},\mu^{k-1})$ or $(\mu^{k}, \omu^{k-1})$  where the inequalities from Definition \ref{def:characterization} hold with equality. 
 An $\alpha_0$-active diagonal is an active diagonal parallel to $\alpha_0$, and an $\alpha_1$-active diagonal is one parallel to $\alpha_1$. 
\end{Definition}

\begin{Remark} \label{rem:cut-into-quads}
Let $P$ be an MV polytope, and cut P along any active diagonal. It is immediate from Definition \ref{def:characterization} that both halves are themselves MV polytopes. 
\end{Remark}

\begin{Definition}
 A {\it complete system of diagonals} $S$ is a choice of one of
 $\{(\omu_{k}, \mu_{k-1}),(\mu_{k}, \omu_{k-1}) \}$ and one of $\{ (\omu^{k},\mu^{k-1}),(\mu^{k}, \omu^{k-1}) \}$ for each $k \geq 2$, which stabilizes in the sense that either $S$ contains $(\omu_{k}, \mu_{k-1})$ for all sufficiently large $k$, or it contains $(\mu_{k}, \omu_{k-1}) $ for all sufficiently large $k$, and similarly for the upper diagonals.
\end{Definition}

\begin{Proposition}  \label{prop:types}
For any complete system $S$ of diagonals, there is some $P \in \MV$ where all diagonals in $S$ are active, but no other diagonals are active. 
\end{Proposition}

\begin{proof}
Proceed by induction on the number of times the type of diagonal ($\alpha_0$ versus $\alpha_1$) in $S$ changes as you move from the bottom to the top. We will prove the statement along with the extra assumption that the polytope $P$ can be chosen such that all the active diagonals below the first time the type of diagonal changes coincide. 
If the number of changes is 0, one can take $P$ to be a line segment parallel to $\alpha_0$ or $\alpha_1$. 

Assume the number of changes is at least 1, and without loss of generality assume $(\mu_2, \omu_1) \in S$. Let $d$ be the lowest $\alpha_0$ diagonal in $S$. Let $S'$ be the complete system of diagonals obtained from $S$ by replacing all $\alpha_1$ diagonals in $S$ below $d$ with $\alpha_0$ diagonals. By induction, we can find the required polytope  $P'$ for $S'$ such that all active diagonals of $P'$ below $d$ coincide with~$d$.  

There are three cases, based on whether $d$ is of the form $(\omu_k, \mu_{k-1})$, $(\omu^\infty, \mu^\infty)$ or $(\mu^{k+1}, \omu^k)$. In each case, one can glue a quadrilateral or a triangle at the bottom of $P'$ to obtain the desired $P$, as shown in Figure \ref{fig:append-quad}. This can be done such that all the edge length are rational, and then we can rescale to get an element of $\MV$ satisfying the required conditions. 
\end{proof}

\begin{figure}[ht]

\begin{tikzpicture}[yscale=0.15, xscale=0.45]

\draw 
(-1,9.4) node {$\bullet$}
(-.9,7.2) node {$\bullet$}
(0,0) node {$\bullet$}
(4.2,4.2) node {$\bullet$}
(4.7, 9.7) node {$\bullet$};

\draw[line width = 0.03cm, color=gray]
(5.7, 10.7)--(-1.3,3.7);

\draw[line width = 0.05cm] 
(0,0)--(4.2,4.2)
(4.2,4.2)--(-1,9.4)
(0,0)--(-.9,7.2)--(-1,9.4)
(4.2, 4.2)--(4.7, 9.7)
;

\draw (0.8,9.3) node {$d$}
(6.4, 4.1) node {$\mu_{k-1}=\mu_1$}
(-2, 6.9) node {$\omu_{k-1}$}
(-1.8, 10.3) node {$\omu_k$}
(5.4, 8.8) node {$\mu_k$} 
(.2,-1.6) node {$\omu_{k-2}=\mu_0$}
;

\end{tikzpicture}
\hspace{0.5in}
\begin{tikzpicture}[yscale=0.15, xscale=0.45]

\draw 
(0,0) node {$\bullet$}
(5,5) node {$\bullet$}
(0,10) node {$\bullet$}
;

\draw[line width = 0.05cm] 
(0,0)--(5,5)
(5,5)--(0,10)
(0,0)--(0,10)
;

\draw (2.6,9) node {$d$}
(-1, 0) node {$\omu_{\infty}$}
(-1, 10) node {$\omu^{\infty}$}
(7, 5) node {$\mu_\infty=\mu^\infty$}
;

\end{tikzpicture}

\caption{\label{fig:append-quad}
The inductive step for the proof of Proposition \ref{prop:types}. In the left diagram, $d=(\omu_{k}, \mu_{k-1})$. Inductively we can find a polytope for $S'$, which we place above diagonal $d$. In the left diagram, since $(\mu_k,\omu_{k-1})$ was not active in the original polytope, $\omu_k$ is strictly above the line through $\mu_k$ parallel to $\alpha_1$. Thus we can append the shown quadrilateral, where $\omu_{k-1}$ is chosen to be sufficiently close to $\omu_k$, but not equal to $\omu_k$. In the second figure, we can append a triangle, where the partition associated to the vertical edge is a single part. The case when $d$ of of type $(\mu^{k+1}, \omu^k)$ is similar to the first case. 
}
\end{figure}
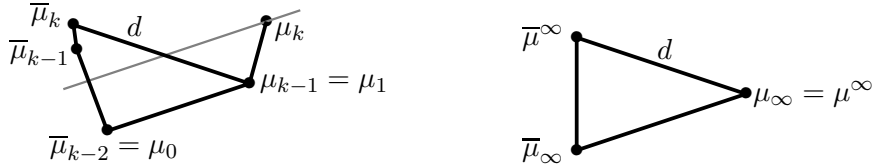

\begin{Definition}
A Lusztig datum is a choice of a tuple ${\bf a} = (a_k, \lambda_k, a^k)_{k \in \bn}$ of non-negative real numbers such that
\begin{enumerate}
\item  for all sufficiently large $k$, $a_k=a^k=\lambda_k=0$, and 
\item $\lambda_1 \geq \lambda_2 \geq \cdots$.
\end{enumerate}
The weight $\wt({\bf a})$ of a Lusztig datum ${\bf a}$ is 
$$\wt({\bf a}):= \sum_k a_k (\alpha_1+ (k-1) \delta) + |\lambda| \delta + \sum_k a^k (\alpha_0+ (k-1) \delta).$$
\end{Definition}

\begin{Theorem} \label{th:Lusztig-data}
For each Lusztig datum ${\bf a}$,
there is a unique $\asl_2$ MV polytope $P_{\bf a}$ whose right Lusztig data is given by ${\bf a}$. Furthermore $P_{\bf a}$ is integral if and only if all $a_k, \lambda_k$ and $a^k$ are integers. 
\end{Theorem}

\begin{Remark}
We often use the notation ${\bf \oa}$ to denote the left Lusztig data of the polytope $P_{\bf a}$ specified by Theorem \ref{th:Lusztig-data}.
\end{Remark}

\begin{Remark}
In finite type, the lengths of the edges along one side of an MV polytope (or, in rank greater than two, along certain paths from the bottom to the top vertex) is the data used by Lusztig to parameterize that element of $B(-\infty)$. Lusztig's construction uses a longest element of the Weyl group, so does not generalize immediately to $\asl_2$. However, \cite{Damiani:1993} (see also \cite{BCP:1999}) gives an analogue of Lusztig parameterization for $\asl_2$. It would be interesting to determine if our combinatorics can be related to this algebraic construction.
\end{Remark}

\subsection{Proof of Theorem \ref{th:Lusztig-data}}

Our proof is effective: we present
a recursive algorithm to construct the polytope $P_{\mathbf a}$ by reducing to cases $\mathbf a'$ where the height of $\mathbf a'$ is less than the height of $\mathbf a$ (here the height of a Lusztig datum ${\bf b}$ is defined to be $(\wt({\bf b}), \omega_0+ \omega_1)$).
We begin by considering some simple Lusztig data which will serve as building blocks.

\begin{Lemma} \label{le:a_ka_k+1}
Theorem~\ref{th:Lusztig-data} holds for a Lusztig data of the form
$(0,\ldots,0,a_k,a_{k+1},0,\ldots,0)$, and the corresponding left
Lusztig data is $(\oa_1,0,\ldots,0,\oa^1)$, with
$\oa_1=(k-1)a_k+ka_{k+1}$ and $\oa^1=ka_k+(k+1)a_{k+1}$.
\begin{center}
\begin{tikzpicture}[xscale=0.3, yscale=0.1]
\draw 
(0,0) node {$\bullet$}
(1,5) node {$\bullet$}
(2,12) node {$\bullet$}
(-5,5) node {$\bullet$};

\draw [line width = 0.05cm] 
(0,0)--(1,5)
(1,5)--(2,12)
(2,12)--(-5,5)
(0,0)--(-5,5);

\draw
(1.6,2) node {$a_k$}
(3.2,8.4) node {$a_{k+1}$}
(-2,10.7) node {$\oa^1$}
(-2.8,0.3) node {$\oa_1$}
(-6,5) node {$\omu_k$};
\end{tikzpicture}
\end{center}
\end{Lemma}

\begin{Lemma} \label{le:a_1a^1}
Theorem~\ref{th:Lusztig-data} holds for a Lusztig data of the form
$(a_1,0,\ldots,0,a^1)$. Moreover, for the MV polytope with this
right Lusztig data:
\begin{enumerate}
\item \label{a_1a^1:LeftLuszData}
The left Lusztig data is given as follows:
\begin{itemize}
\item
If $a_1=a^1$, then $\oa_k=\oa^k=0$ for all $k$ and
$\overline\lambda=(a_1\geq0\geq0\geq\ldots)$ (left picture below).
\item
If $a_1<a^1$, then there is an integer $r\geq1$ such that
$\frac{r-1}r\leq\frac{a_1}{a^1}\leq\frac r{r+1}$ and
$$\begin{pmatrix}\oa_r\\\oa_{r+1}\end{pmatrix}=
\begin{pmatrix}r&-(r+1)\\-(r-1)&r\end{pmatrix}
\begin{pmatrix}a^1\\a_1\end{pmatrix};$$
all the other $\oa_k$ are zero, as are all the $\oa^k$
and all the $\overline\lambda_k$(right picture below).
\item
The case $a_1>a^1$ can be obtained from the previous one by
vertical reflection.
\end{itemize}
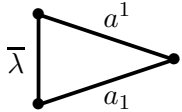
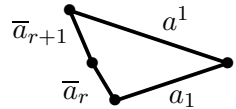
\begin{center}
\begin{tikzpicture}[xscale=0.3, yscale=0.1]
\begin{scope}
\draw 
(0,0) node {$\bullet$}
(0,12) node {$\bullet$}
(6,6) node {$\bullet$};

\draw [line width = 0.05cm] 
(0,0)--(0,12)--(6,6)--cycle;

\draw
(3.5,12) node {$a^1$}
(3.5,0.6) node {$a_1$}
(-1,6) node {$\overline\lambda$};
\end{scope}
\begin{scope}[xshift=22cm, yshift=0cm]
\draw 
(0,0) node {$\bullet$}
(-1,5) node {$\bullet$}
(-2,12) node {$\bullet$}
(5,5) node {$\bullet$};

\draw [line width = 0.05cm] 
(0,0)--(-1,5)
(-1,5)--(-2,12)
(-2,12)--(5,5)
(0,0)--(5,5);

\draw
(2.7,11) node {$a^1$}
(3,0.3) node {$a_1$}
(-1.7,1.2) node {$\oa_r$}
(-3.3,9.1) node {$\oa_{r+1}$};
\end{scope}
\end{tikzpicture}
\end{center}
\item \label{a_1a^1:Vertices}
Let $k\geq2$. In order that $\mu^0=\omu_k$, it is necessary
and sufficient that $ka_1\leq(k-1)a^1$, and then
$\oa_k=\max(0,(k-1)a_1-(k-2)a^1)$.
In order that $\mu_0=\omu_{k-2}$, it is necessary and
sufficient that $(k-2)a^1\leq(k-1)a_1$, and then
$\oa_{k-1}=\max(0,(k-1)a^1-ka_1)$.
\end{enumerate}
\end{Lemma}

\begin{Lemma} \label{le:a_1a_k}
Theorem~\ref{th:Lusztig-data} holds for a Lusztig data of the form
$(a_1,0,\ldots,0,a_k,0,\ldots,0)$. If $a_1\geq(k-2)a_k$, then
the left Lusztig data is $(0,\ldots,0,\oa_{k-1},0,\ldots,0,\oa^1)$,
with $\oa_{k-1}=a_k$ and $\oa^1=a_1+2a_k$ (left picture below).
If $a_1\leq(k-2)a_k$, then $\oa^1=ka_k$ and the MV polytope is
obtained by stacking a triangle above the MV polytope with right
Lusztig data $(a_1,0,\ldots,0,(k-1)a_k)$ (right picture
below). In any case, $\oa^1=\max(a_1+2a_k,ka_k)$.
\begin{center}
\begin{tikzpicture}[xscale=0.54, yscale=0.18]
\begin{scope}
\draw 
(0,0) node {$\bullet$}
(3,3) node {$\bullet$}
(4,8) node {$\bullet$}
(-1,3) node {$\bullet$};

\draw[line width = 0.05cm] 
(0,0)--(3,3)
(3,3)--(4,8)
(4,8)--(-1,3)
(-1,3)--(0,0);

\draw
(1.6,0) node {$a_1$}
(4,5) node {$a_k$}
(1.6,7.4) node {$\oa^1$}
(-1.3,0.7) node {$\oa_{k-1}$};
\end{scope}
\begin{scope}[xshift=11.5cm, yshift=0cm]
\draw
(0,0) node {$\bullet$}
(-1,1) node {$\bullet$}
(-2,4) node {$\bullet$}
(1,1) node {$\bullet$}
(2,8) node {$\bullet$};

\draw[line width = 0.05cm]
(1,1)--(0,0)--(-1,1)--(-2,4)--(2,8)--(1,1)--(-2,4);

\draw
(.8,-1) node {$a_1$}
(2.2,5) node {$a_k$}
(2.2,1) node {$\mu_{k-1}$}
(-2.7,4) node {$\omu_k$}
(0,7.6) node {$\oa^1$};
\end{scope}
\end{tikzpicture}
\end{center}
\end{Lemma}

\begin{Lemma} \label{le:a_ka^1}
Theorem~\ref{th:Lusztig-data} holds for a Lusztig datum of the form
$(0,\ldots,0,a_k,0,\ldots,0,a^1)$, where $k\geq2$. If $a^1\geq(k+1)a_k$,
then the left Lusztig data is $(\oa_1,0,\ldots,0,\oa_{k+1},0,\ldots,0)$,
with $\oa_1=a^1-2a_k$ and $\oa_{k+1}=a_k$ (left picture below). If
$a^1\leq(k+1)a_k$, then $\oa_1=(k-1)a_k$ and the MV polytope is
obtained by stacking the MV polytope with right Lusztig data
$(ka_k,0,\ldots,0,a^1)$ above a triangle (right picture below).
In any case, $\oa_1=\max(a^1-2a_k,(k-1)a_k)$.
\begin{center}
\begin{tikzpicture}[xscale=0.54, yscale=0.18]
\begin{scope}
\draw 
(0,0) node {$\bullet$}
(-3,3) node {$\bullet$}
(-4,8) node {$\bullet$}
(1,3) node {$\bullet$};

\draw[line width = 0.05cm] 
(0,0)--(-3,3)
(-3,3)--(-4,8)
(-4,8)--(1,3)
(1,3)--(0,0);

\draw
(-1.6,-0.1) node {$\oa_1$}
(-4.4,5.4) node {$\oa_{k+1}$}
(-3.5,2.1) node {$\omu_k$}
(-1,7.3) node {$a^1$}
(1,0.6) node {$a_k$}
(-4.2,10) node {$\omu_{k+1}$};
\end{scope}
\begin{scope}[xshift=9cm, yshift=0cm]
\draw
(0,0) node {$\bullet$}
(1.6,4.8) node {$\bullet$}
(-1.6,8) node {$\bullet$}
(-1.6,1.6) node {$\bullet$};

\draw[line width = 0.05cm] 
(-1.6,8)--(1.6,4.8)--(0,0)--(-1.6,1.6)--(1.6,4.8);
\draw[line width = 0.05cm, dashed]
(-1.6,8) to[out=-100,in=100] (-1.6,1.6);

\draw
(1.2,1.3) node {$a_k$}
(0,8.5) node {$a^1$}
(-1,-0.8) node {$\oa_1$}
(-2.2,1.5) node {$\omu_k$}
(2.3,4.5) node {$\mu_k$};
\end{scope}
\end{tikzpicture}
\end{center}
\end{Lemma}

\begin{Lemma} \label{le:basecases}
Theorem~\ref{th:Lusztig-data} holds for a Lusztig datum of the form
$(a_1,0,\ldots,0,\lambda,0,\ldots,0,a^1)$.
\end{Lemma}

\begin{proof}[Proof of Lemmas \ref{le:a_ka_k+1}, \ref{le:a_1a^1}, \ref{le:a_1a_k}, \ref{le:a_ka^1}, and \ref{le:basecases}]

These are all proven by elementary arguments. As an example,
we show Lemma~\ref{le:a_1a^1}. We also provide some details for Lemma \ref{le:basecases}, since this requires a little more care than the others.

{\bf Proof for Lemma~\ref{le:a_1a^1}:}
We first notice that if the $\alpha_0$-diagonal
$(\omu^\infty,\mu^\infty)$ is active, then
$\omu^\infty=\mu^0$, and that if the $\alpha_1$-diagonal
$(\mu^\infty,\omu^\infty)$ is active,
then $\omu^\infty=\mu_0$. So $\omu^\infty\in\{\mu^0,\mu_0\}$, 
and likewise $\omu_\infty\in\{\mu^0,\mu_0\}$.

Consider the case $a_1=a^1$. If we had $\omu^\infty\neq\mu^0$,
then on the one hand, $\mu^0-\omu^\infty$ would be of the form
$x\alpha_1+y\delta$ with $x>0$, and on the other hand, $\omu^\infty$
would be equal to $\mu_0$: this is clearly impossible, because
$\mu^0-\mu_0=a_1\delta$. Therefore $\omu^\infty=\mu^0$. Likewise,
$\omu_\infty=\mu_0$. Condition~(\ref{part:middle1}) in
Definition~\ref{def:characterization} then gives the announced
result.

Consider now the case $a_1<a^1$. This time,
$\omu^\infty=\omu_\infty=\mu^0$. If some vertex $\omu_r$ for $r \geq 2$ is different
from both $\mu_0$ and $\mu^0$, then neither $(\omu_r,\mu_{r-1})$
nor $(\mu_{r+1},\omu_r)$ is active, so $(\mu_r,\omu_{r-1})$ and
$(\omu_{r+1},\mu_r)$ are active, which forces $\omu_{r-1}=\mu_0$
and $\omu_{r+1}=\mu^0$, and by a similar argument the same conclusion holds when $r=1$.
A straightforward calculation concludes the
proof of (\ref{a_1a^1:LeftLuszData}). The statement
(\ref{a_1a^1:Vertices}) is an immediate corollary of
(\ref{a_1a^1:LeftLuszData}).

{\bf Sketch of proof for Lemma \ref{le:basecases}:}
Here we need to distinguish between four cases.
\begin{itemize}
\item
If $a_1=a^1\geq\lambda_1$, then $\omu_\infty=\mu_0$ and
$\omu^\infty=\mu^0$, and therefore the polytope $P_{\mathbf a}$
is a trapezoid with left decoration $\overline\lambda=(a_1\geq
\lambda_1\geq\lambda_2\geq\cdots)$ (left picture below).
\item
If $\lambda_1>\max(a_1,a^1)$, then the $\alpha_0$-diagonal
$(\omu_\infty,\mu_\infty)$ and the $\alpha_1$-diagonal
$(\mu^\infty,\omu^\infty)$ are
active, and the polytope $P_{\mathbf a}$ is obtained by stacking
the MV polytope with right Lusztig data $(a_1,0,\ldots,0,\lambda_1)$,
a trapezoid with decoration $\overline\lambda=(\lambda_2\geq\lambda_3
\geq\cdots)$, and the MV polytope with right Lusztig data
$(\lambda_1,0,\ldots,0,a^1)$ (central picture below).
\item
If $a^1>a_1$ and $a^1\geq\lambda_1$, then $\omu^\infty=\mu^0$, the
$\alpha_0$-diagonal $(\omu_\infty,\mu_\infty)$ is active, and the polytope
$P_{\mathbf a}$ is obtained by stacking the MV polytope with right
Lusztig data $(a_1,0,\ldots,0,a^1)$ and the parallelogram with
left decoration $\lambda$ (right picture below).
\item
The case $a_1>a^1$ and $a_1\geq\lambda_1$ can be obtained from the
previous one by vertical reflection.
\end{itemize}
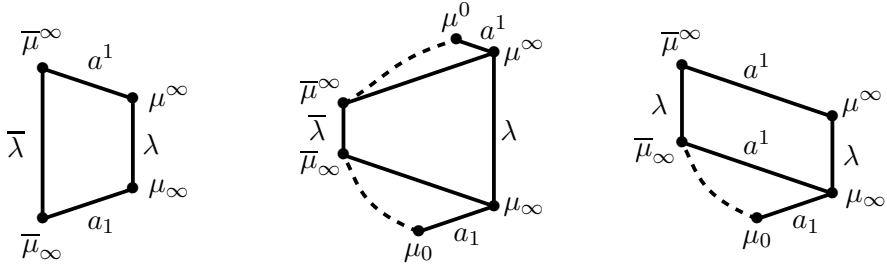
\begin{center}
\begin{tikzpicture}
\begin{scope}[xscale=0.6, yscale=0.2]
\draw
(0,0) node {$\bullet$}
(2,2) node {$\bullet$}
(2,8) node {$\bullet$}
(0,10) node {$\bullet$};

\draw[line width = 0.05cm] 
(0,0)--(2,2)--(2,8)--(0,10)--cycle;

\draw
(0,-1.8) node {$\omu_\infty$}
(2.8,2) node {$\mu_\infty$}
(2.8,8) node {$\mu^\infty$}
(0,11.8) node {$\omu^\infty$}
(1.3,-0.4) node {$a_1$}
(2.4,5) node {$\lambda$}
(-0.6,5) node {$\overline\lambda$}
(1.3,10.8) node {$a^1$};
\end{scope}
\begin{scope}[xscale=0.5, yscale=0.17, xshift=10cm, yshift=-1cm]
\draw 
(0,0) node {$\bullet$}
(2,2) node {$\bullet$}
(-2,6) node {$\bullet$}
(-2,10) node {$\bullet$}
(2,14) node {$\bullet$}
(1,15) node {$\bullet$};

\draw[line width = 0.05cm] 
(0,0)--(2,2)--(2,14)--(1,15)
(2,2)--(-2,6)--(-2,10)--(2,14);
\draw[line width = 0.05cm, dashed]
(0,0) to[out=130,in=-80] (-2,6)
(-2,10) to[out=60,in=-140] (1,15); 

\draw
(0,-1.6) node {$\mu_0$}
(2.8,2) node {$\mu_\infty$}
(-2.6,5.4) node {$\omu_\infty$}
(-2.6,11) node {$\omu^\infty$}
(2.8,14) node {$\mu^\infty$}
(1,16.8) node {$\mu^0$}
(1.3,-0.4) node {$a_1$}
(2.4,8) node {$\lambda$}
(-2.7,8.2) node {$\overline\lambda$}
(2,16) node {$a^1$};
\end{scope}
\begin{scope}[xscale=0.5, yscale=0.17, xshift=19cm, yshift=0cm]
\draw
(0,0) node {$\bullet$}
(2,2) node {$\bullet$}
(2,8) node {$\bullet$}
(-2,12) node {$\bullet$}
(-2,6) node {$\bullet$};

\draw[line width = 0.05cm] 
(0,0)--(2,2)--(2,8)--(-2,12)--(-2,6)--(2,2);
\draw[line width = 0.05cm, dashed]
(0,0) to[out=130,in=-80] (-2,6);

\draw
(0,-1.6) node {$\mu_0$}
(2.9,1.6) node {$\mu_\infty$}
(-2.7,5.4) node {$\omu_\infty$}
(-2,13.8) node {$\omu^\infty$}
(2.8,9) node {$\mu^\infty$}
(1.3,-0.4) node {$a_1$}
(2.5,5) node {$\lambda$}
(-2.6,9) node {$\lambda$}
(0,6) node {$a^1$}
(0,12) node {$a^1$};
\end{scope}
\end{tikzpicture}
\end{center}
\end{proof}

The existence and uniqueness of $P_{\mathbf a}$ will be established
alongside the following properties.

\begin{Proposition} \label{pr:a1not0}
Let $\mathbf a$ be a Lusztig datum such that $a_1>0$ and set
$\mathbf b=(0,a_2,a_3,\ldots,a^1)$. Let $k\geq2$ be such that
$a_2=\cdots=a_{k-1}=0$. Then $(k-2)\ob_1\leq(k-1)a_1$ if and
only if $\oa_1=\cdots=\oa_{k-2}=0$, and if these assertions hold
true, then $\oa_{k-1}=\max(a_k,(k-1)\ob_1-ka_1)$.
\end{Proposition}

\begin{Proposition} \label{pr:a10}
Let $\mathbf a$ be a Lusztig datum and let $k\geq2$
such that $a_1=\cdots=a_{k-1}=0$. Set
$\mathbf d=(0,\ldots,0,a_{k+1},a_{k+2},\ldots,a^1)$.
Then $\oa_1=\max(\od_1-2a_k,(k-1)a_k+ka_{k+1})$
and $\oa_2=\cdots=\oa_k=0$.
\end{Proposition}

\begin{proof}[Proof of Theorem~\ref{th:Lusztig-data}, Proposition~\ref{pr:a1not0} and Proposition~\ref{pr:a10}]
We prove these simultaneously by induction on the height $(\wt(\mathbf a), \omega_0+ \omega_1)$ of the Lusztig data $\mathbf a$, the case where the height is $0$ being trivial. The inductive step breaks into three cases, each of which 
requires several arguments. 
\medbreak

\noindent\textbf{Induction step when $a_1>0$ and there exists $\ell\geq2$ such that $a_\ell>0$:}

Denote by $k$ the smallest $\ell\geq2$ such that $a_\ell > 0$.
Set $\mathbf b=(0,\ldots,0,a_k,a_{k+1},\ldots,a^1)$,
$\mathbf c=(c_1,0,\ldots,0,a_{k+1},a_{k+2},\ldots,a^1)$,
where $c_1=\max(a_1+2a_k,ka_k)$,
and $\mathbf d=(0,\ldots,0,a_{k+1},a_{k+2},\ldots,a^1)$.
By induction, we know that $P_{\mathbf b}$ and $P_{\bf c}$ exist and are
unique. By the inductive assumption, Proposition~\ref{pr:a10} gives
$\ob_1=\max(\od_1-2a_k,(k-1)a_k+ka_{k+1})$ and
$\ob_2=\cdots=\ob_k=0$, where $\overline{\mathbf b}$ is the left Lusztig data of $P_{\bf b}$. 

\medbreak

By definition, at least one of the diagonals $(\omu_k,\mu_{k-1})$
or $(\mu_k,\omu_{k-1})$ is active in any MV polytope with right Lusztig data $\mathbf a$, which gives two possibilities:

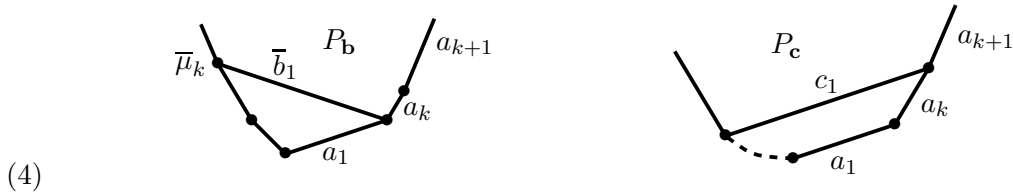
\begin{equation} \label{mc1}
\begin{tikzpicture}[yscale=0.15, xscale=0.45]
\begin{scope}
\draw [line width = 0.05cm](3,3)--(-2,8);
\draw(1.5,-0.15) node {$a_1$};
\draw(3.9,3.8) node {$a_k$};
\draw(5.3,9.5) node {$a_{k+1}$};
\draw(0,8) node {$\ob_1$};
\draw(1.6,10) node {$P_{\mathbf b}$};
\draw(-2.8,8) node {$\omu_k$};

\draw (0,0) node {$\bullet$};
\draw (3,3) node {$\bullet$};
\draw (3.5,5.5) node {$\bullet$};
\draw (-1,3) node {$\bullet$};
\draw (-2,8) node {$\bullet$};

\draw [line width = 0.05cm](0,0)--(3,3);
\draw [line width = 0.05cm](3,3)--(3.5,5.5);
\draw [line width = 0.05cm](3.5,5.5)--(4.4,12);
\draw [line width = 0.05cm](0,0)--(-1,3);
\draw [line width = 0.05cm](-1,3)--(-2,8);
\draw [line width = 0.05cm](-2,8)--(-2.5,11.5);
\end{scope}
\begin{scope}[xshift=14cm, yshift=-1.4cm]
\draw [line width = 0.05cm](-1,3)--(5,9);
\draw(2.5,0.4) node {$a_1$};
\draw(5.2,5.5) node {$a_k$};
\draw(6.7,11.5) node {$a_{k+1}$};
\draw(2,7.6) node {$c_1$};
\draw(0.8,11) node {$P_{\mathbf c}$};

\draw (1,1) node {$\bullet$};
\draw (4,4) node {$\bullet$};
\draw (5,9) node {$\bullet$};
\draw (-1,3) node {$\bullet$};

\draw [line width = 0.05cm](1,1)--(4,4);
\draw [line width = 0.05cm](4,4)--(5,9);
\draw [line width = 0.05cm](5,9)--(5.8,14.6);
\draw [line width = 0.05cm, dashed] (1,1) .. controls (-0.2,1.2) and (-0.2,1.2) .. (-1,3);
\draw [line width = 0.05cm](-1,3)--(-2.5,10.5);
\end{scope}
\end{tikzpicture}
\end{equation}

\noindent\textit{Uniqueness in Theorem~\ref{th:Lusztig-data}:}
Assume $Q$ is an MV polytope with right Lusztig data $\bf a$. 

\begin{itemize}
\item
If $(\omu_k,\mu_{k-1})$ is active in $Q$, then we can cut $Q$ in
two along this diagonal (left picture). The polytope above the
diagonal is necessarily $P_{\mathbf b}$, and so below the diagonal,
we have the MV polytope with right Lusztig data
$(a_1,0,\ldots,0,\ob_1)$, which is also fully determined
(Lemma~\ref{le:a_1a^1}). Thus $Q$ is fully determined. Since $Q$ is MV, in the bottom polytope, the
right upper edge must coincide with the diagonal $(\omu_k,\mu_{k-1})$, so, by Lemma~\ref{le:a_1a^1}~(\ref{a_1a^1:Vertices}), $ka_1\leq(k-1)\ob_1$ and $\oa_k=\max(0,(k-1)a_1-(k-2)\ob_1)$. Since
$$\omu_{k-1}=\mu_1+\ob_1\alpha_0-\oa_k(\alpha_0+(k-1)\delta)\quad
\text{and}\quad\mu_k=\mu_1+a_k(\alpha_1+(k-1)\delta),$$
the condition $(\omega_0,\mu_k-\omu_{k-1})\leq0$ gives
$(k-1)a_k-\ob_1+k\oa_k\leq0$, so that
$$\ob_1\geq(k-1)a_k+k((k-1)a_1-(k-2)\ob_1),$$
which rearranges to $(k-1)\ob_1\geq ka_1+a_k$.
\item
If $(\omu_k,\mu_{k-1})$ is inactive in $Q$, then $(\mu_k,\omu_{k-1})$
is active, and we can cut $Q$ along this latter diagonal (right picture).
The polytope below this diagonal is necessarily the polytope with right Lusztig
data $(a_1,0,\ldots,0,a_k,0,\ldots)$, which is fully determined by
Lemma~\ref{le:a_1a_k}. The height of ${\bf c}$ is less than the height of ${\bf a}$, so by the inductive hypothesis the MV polytope $P_{\bf c}$ above the diagonal exists and is unique. Thus $Q$ is fully determined. Since we assume $Q$ is MV and $(\omu_k, \mu_{k-1})$ is inactive, we have $(\omega_1, \omu_k-\mu_{k-1})<0$.
Since $\omu_k=\mu_k-c_1\alpha_1+\oc_k(\alpha_0+(k-1)\delta)$ and $\mu_{k-1}=\mu_k-a_k(\alpha_1+(k-1)\delta)$, this gives $ka_k+(k-1)\oc_k<c_1$. In particular, we see that $c_1>ka_k$, and so by the definition of ${\bf c}$ we have $c_1=a_1+2a_k$.

Now the bottom vertex of $P_{\mathbf c}$ is its vertex $\omu_{k-1}$,
so $\oc_1=\cdots=\oc_{k-1}=0$. By Proposition~\ref{pr:a1not0}
applied to $\mathbf c$, we have
$\oc_k=\max(c_{k+1},k\od_1-(k+1)c_1)$. It is then easy to
rewrite the inequality $ka_k+(k-1)\oc_k<c_1$ as the system
$$(k-1)\od_1<ka_1+(2k-1)a_k\quad\text{and}
\quad(k-2)a_k+(k-1)a_{k+1}<a_1.$$
Remembering $\ob_1=\max(\od_1-2a_k,(k-1)a_k+ka_{k+1})$,
we obtain $(k-1)\ob_1<ka_1+a_k$.
\end{itemize}
It follows that there is always at most one $Q \in \MV$ with right Lusztig data
$\mathbf a$, where the diagonal $(\omu_k,\mu_{k-1})$ must be active if $(k-1)\ob_1\geq ka_1+a_k$ and inactive if 
$(k-1)\ob_1< ka_1+a_k$.
\medbreak

\noindent\textit{Existence in Theorem~\ref{th:Lusztig-data}:}
\begin{itemize}
\item
Assume first that $(k-1)\ob_1\geq ka_1+a_k$. Then $(k-1)\ob_1\geq ka_1$,
so the polytope with right Lusztig datum $(a_1,0,\ldots,0,\ob_1)$ is
a quadrilateral, whose upper vertex is $\omu_k$
(Lemma~\ref{le:a_1a^1}~(\ref{a_1a^1:Vertices})).
Since $\ob_2=\cdots=\ob_k=0$, we can place the polytope $P_{\mathbf b}$
above this quadrilateral and obtain a GGMS polytope.
By Lemma~\ref{le:a_1a^1}~(\ref{a_1a^1:Vertices}), we have
$\oa_k=\max(0,(k-1)a_1-(k-2)\ob_1)$; combining $\ob_1\geq(k-1)a_k$
with $(k-1)\ob_1\geq ka_1+a_k$, we get
$\ob_1-(k-1)a_k\geq k\oa_k$. As in the uniqueness part of the proof,
we then compute $(\omega_0,\mu_k-\omu_{k-1})=(k-1)a_k-\ob_1+k\oa_k$.
This is nonpositive, hence our candidate polytope is MV.
\item
Assume now that $(k-1)\ob_1\leq ka_1+a_k$. Then
$$(k-1)\od_1\leq ka_1+(2k-1)a_k\leq k(a_1+2a_k)\leq kc_1.$$
Proposition~\ref{pr:a1not0} applied to $\mathbf c$
then guarantees that $\oc_1=\cdots=\oc_{k-1}=0$ and that
$\oc_k=\max(c_{k+1},k\od_1-(k+1)c_1)$. In particular, the bottom
vertex of $P_{\mathbf c}$ is $\omu_{k-1}$. Moreover,
$$(k-1)((k-1)a_k+ka_{k+1})\leq(k-1)\ob_1\leq ka_1+a_k$$
gives $(k-1)^2a_k\leq ka_1+a_k$, and then $a_1\geq(k-2)a_k$. Therefore
by Lemma~\ref{le:a_1a_k} the MV polytope with right Lusztig data
$(a_1,0,\ldots,0,a_k,0,\ldots)$ 
can be placed
below
$P_{\mathbf c}$ to obtain a GGMS polytope.

The inequality $a_1\geq(k-2)a_k$ gives $c_1=a_1+2a_k$.
From $(k-1)\od_1\leq ka_1+(2k-1)a_k$ and
$(k-1)((k-1)a_k+ka_{k+1})\leq ka_1+a_k$, we deduce
$ka_k+(k-1)\oc_k<c_1$. It follows that
$(\omega_0,\omu_k-\mu_{k-1})\leq0$, which
implies that our candidate polytope is MV.
\end{itemize}
\medbreak

\noindent\textit{Proposition~\ref{pr:a1not0}:}

Let $\ell\geq2$ be such that $a_2=\cdots=a_{\ell-1}=0$.
Since $a_k>0$, we necessarily have $\ell\leq k$.
\begin{itemize}
\item
Assume first that $(k-1)\ob_1\geq ka_1+a_k$. Then $P_{\mathbf a}$
is given by the left diagram in Equation \eqref{mc1}. The numbers $\oa_1$, \dots,
$\oa_\ell$ for $P_{\mathbf a}$ come from the quadrilateral at the
bottom. Applying Lemma~\ref{le:a_1a^1}~(\ref{a_1a^1:Vertices}) to
this quadrilateral, we see that $\oa_1=\cdots=\oa_{\ell-2}=0$ if and
only if $(\ell-2)\ob_1\leq(\ell-1)a_1$, and if these assertions hold
true, then $\oa_{\ell-1}=\max(0,(\ell-1)\ob_1-\ell a_1)$.

To complete the proof of Proposition~\ref{pr:a1not0}, it now
suffices to check that
$\max(0,(\ell-1)\ob_1-\ell a_1)=\max(a_\ell,(\ell-1)\ob_1-\ell a_1)$.
This equality is shown by looking separately at the cases $\ell<k$
and $\ell=k$. In the first case, one simply notices that $a_\ell=0$.
In the second case, one uses the assumption $(k-1)\ob_1\geq ka_1+a_k$.
\item
Assume now that $(k-1)\ob_1\leq ka_1+a_k$. Then $P_{\mathbf a}$
is given by the right diagram in Equation \eqref{mc1}, so we have
$\oa_1=\cdots=\oa_{k-2}=0$ and $\oa_{k-1}=a_k$. Furthermore, we
saw in the existence proof above that $a_1\geq(k-2)a_k$, which implies
$$(k-1)\ob_1\leq ka_1+\frac{a_1}{k-2}=\frac{(k-1)^2}{k-2}a_1,$$
whence $\ob_1/a_1\leq(k-1)/(k-2)\leq(\ell-1)/(\ell-2)$.
Therefore both assertions $\oa_1=\cdots=\oa_{\ell-2}=0$ and
$(\ell-2)\ob_1\leq(\ell-1)a_1$ hold true.

It remains to check
that $\oa_{\ell-1}=\max(a_\ell,(\ell-1)\ob_1-\ell a_1)$.
If $\ell\leq k-1$, this comes from the fact
that $\oa_{\ell-1}=a_\ell=0$ and that
$\ob_1/a_1\leq(k-1)/(k-2)\leq\ell/(\ell-1)$.
If $\ell=k$, this comes from the equality
$\oa_{k-1}=a_k$ and from the assumption $(k-1)\ob_1\leq ka_1+a_k$.
\end{itemize}
\medbreak

\noindent\textbf{Induction step when $a_1=0$ and $a_\ell>0$ for some $\ell\geq2$:}

Denote by $k$ the smallest value such that $a_k \neq 0$. 
We set $\mathbf b=(0,\ldots,0,a_{k+1},a_{k+2},\ldots,a^1)$,
$\mathbf c=(c_1,0,\ldots,0,a_{k+2},a_{k+3},\ldots,a^1)$,
where $c_1=ka_k+(k+1)a_{k+1}$,
and $\mathbf d=(0,\ldots,0,a_{k+2},a_{k+3},\ldots,a^1)$.
By induction, we know that $P_{\mathbf b}$ exists and is
unique, and we denote its left Lusztig datum by $\overline{\mathbf b}$.
Moreover, Proposition~\ref{pr:a10} gives
$\ob_1=\max(\od_1-2a_{k+1},ka_{k+1}+(k+1)a_{k+2})$ and
$\ob_2=\cdots=\ob_{k+1}=0$.
\medbreak

By definition, at least one of the diagonals $(\mu_{k+1},\omu_k)$
or $(\omu_{k+1},\mu_k)$ is active in any MV polytope with right Lusztig data $\mathbf a$, giving two possibilities:
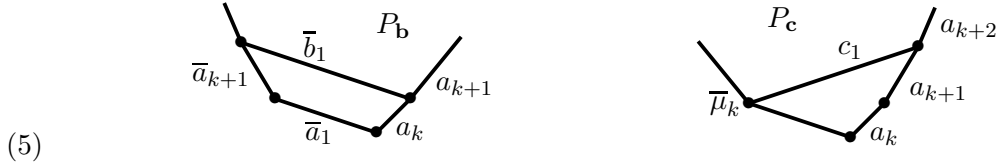
\begin{equation} \label{mc2}
\begin{tikzpicture}[yscale=0.15, xscale=0.45]
\begin{scope}
\draw(-1.7,0) node {$\oa_1$};
\draw(1,0.1) node {$a_k$};
\draw(2.6,4) node {$a_{k+1}$};
\draw(-4.6, 5.1) node {$\oa_{k+1}$};
\draw(.5,9.5) node {$P_{\mathbf b}$};
\draw(-1.8,7.7) node {$\ob_1$};

\draw (0,0) node {$\bullet$};
\draw (1,3) node {$\bullet$};
\draw (-3,3) node {$\bullet$};
\draw (-4,8) node {$\bullet$};

\draw [line width = 0.05cm](0,0)--(1,3);
\draw [line width = 0.05cm](1,3)--(2.5,8.5);
\draw [line width = 0.05cm](0,0)--(-3,3);
\draw [line width = 0.05cm](-3,3)--(-4,8);
\draw [line width = 0.05cm](-4,8)--(-4.5,11.5);
\draw [line width = 0.05cm](1,3)--(-4,8);
\end{scope}
\begin{scope}[xshift=14cm, yshift=-.4cm]
\draw(1,0.1) node {$a_k$};
\draw(2.6,4) node {$a_{k+1}$};
\draw(3.5,9.6) node {$a_{k+2}$};
\draw(-2,10.5) node {$P_{\mathbf c}$};
\draw (0, 7.7) node {$c_1$};
\draw (-3.7,3) node {$\omu_k$};

\draw (0,0) node {$\bullet$};
\draw (1,3) node {$\bullet$};
\draw (2,8) node {$\bullet$};
\draw (-3,3) node {$\bullet$};

\draw [line width = 0.05cm](0,0)--(1,3);
\draw [line width = 0.05cm](1,3)--(2,8);
\draw [line width = 0.05cm](2,8)--(2.5,11.5);
\draw [line width = 0.05cm](0,0)--(-3,3);
\draw [line width = 0.05cm](-3,3)--(-4.5,8.5);
\draw [line width = 0.05cm](-3,3)--(2,8);
\end{scope}
\end{tikzpicture}
\end{equation}

\noindent\textit{Uniqueness in Theorem~\ref{th:Lusztig-data}:} Assume $Q \in \MV$ has right Lusztig data ${\bf a}$.

\begin{itemize}
\item
If $(\mu_{k+1},\omu_k)$ is active in $Q$, then we can cut $Q$ in
two along this diagonal (right picture). The polytope below is
necessarily the polytope with right Lusztig data
$(0,\ldots,0,a_k,a_{k+1},0,\ldots,0)$, which is fully determined
by Lemma~\ref{le:a_ka_k+1}. Its left upper edge has length $c_1$,
and thus the MV polytope above the diagonal is necessarily
$P_{\mathbf c}$. (Notice that the height of $\wt(\mathbf c)$ is
smaller than the height of $\wt(\mathbf a)$, so that $P_{\mathbf c}$
exists and is unique, by the induction hypothesis.) Thus $Q$ is fully determined. Since $Q$ is MV, the lower
vertex of $P_{\mathbf c}$ is $\omu_k$, so 
$\oc_1=\cdots=\oc_k=0$. Hence by Proposition~\ref{pr:a1not0} applied to $P_{\bf c}$ we have $k\od_1\leq(k+1)c_1$ and
$\oc_{k+1}=\max(c_{k+2},(k+1)\od_1-(k+2)c_1)$. Since
$$\omu_{k+1}=\omu_k+\oc_{k+1}(\alpha_0+k\delta)\quad\text{and}\quad
\mu_k=\omu_k+c_1\alpha_1-a_{k+1}(\alpha_1+k\delta),$$
the equation $(\omega_1,\omu_{k+1}-\mu_k)\leq0$ gives
$\oc_{k+1}\leq a_k$. This translates to the system
$$a_{k+2}\leq a_k\quad\text{and}\quad
\od_1\leq(k+1)a_k+(k+2)a_{k+1}.$$
Remembering that $\ob_1=\max(\od_1-2a_{k+1},ka_{k+1}+(k+1)a_{k+2})$,
we conclude that $\ob_1\leq(k+1)a_k+ka_{k+1}$.
\item
If $(\mu_{k+1},\omu_k)$ is inactive in $Q$, then $(\omu_{k+1},\mu_k)$
is active, and we can cut $Q$ along this latter (left picture).
The polytope above the diagonal is necessarily $P_{\mathbf b}$, and
so below the diagonal we have the MV polytope with right Lusztig
data $(0,\ldots,0,a_k,0,\ldots,0,\ob_1)$. The latter is fully
determined by Lemma~\ref{le:a_ka^1}. Thus $Q$ is fully determined, and we even know that
$\oa_1=\max(\ob_1-2a_k,(k-1)a_k)$. Now $(\mu_{k+1},\omu_k)$ is inactive,
so $(\omega_0,\mu_{k+1}-\omu_k)<0$. Since $\omu_k=\oa_1\alpha_0$
and $\mu_{k+1}=a_k(\alpha_1+(k-1)\delta)+a_{k+1}(\alpha_1+k\delta)$,
this gives $\oa_1>(k-1)a_k+ka_{k+1}$, and therefore
$\ob_1>(k+1)a_k+ka_{k+1}$.
\end{itemize}
It follows that there can only ever be at most one $Q \in \MV$ with right Lusztig data $\bf a$, where the diagonal $(\mu_{k+1},\omu_k)$ must be active 
if $\ob_1 \leq(k+1)a_k+ka_{k+1}$, and inactive if $\ob_1>(k+1)a_k+ka_{k+1}$.
\medbreak

\noindent\textit{Existence in Theorem~\ref{th:Lusztig-data}:}
\begin{itemize}
\item
Assume first that $\ob_1\geq(k+1)a_k+ka_{k+1}$. Then a fortiori
$\ob_1\geq(k+1)a_k$, and so, by Lemma~\ref{le:a_ka^1}, the MV
polytope with right Lusztig data $(0,\ldots,0,a_k,0,\ldots,0,\ob_1)$
is the quadrilateral at the bottom on the left diagram in Equation \eqref{mc2}. For
this polytope, the vertex $\omu_{k+1}$ coincides with the top vertex
and we have $\oa_{k+1}=a_k$. Remembering that $\ob_2=\cdots=\ob_{k+1}=0$,
we see that we can place $P_{\mathbf b}$ above this quadrilateral and
get a GGMS polytope. A direct computation then gives
$(\omega_0,\mu_{k+1}-\omu_k)=ka_{k+1}-\ob_1+(k+1)\oa_{k+1}$. This
is nonpositive, hence our candidate polytope is MV.
\item
Assume now that $\ob_1\leq(k+1)a_k+ka_{k+1}$. Then
$\od_1\leq(k+1)a_k+(k+2)a_{k+1}$ and  $a_{k+2}\leq a_k$,
and therefore
$$k\od_1\leq k(k+1)a_k+k(k+2)a_{k+1}\leq(k+1)c_1,$$
where the last inequality comes because $c_1= k a_k + (k+1) a_{k+1}$.
Proposition~\ref{pr:a1not0} applied to $\mathbf c$ leads to
$\oc_1=\cdots=\oc_k=0$ and $\oc_{k+1}=\max(c_{k+2},(k+1)\od_1-(k+2)c_1)$.
Therefore, the bottom vertex of $P_{\mathbf c}$ is $\omu_k$, so
$P_{\mathbf c}$ can be stacked above the MV polytope with right
Lusztig data $(0,\ldots,0,a_k,a_{k+1},0,\ldots,0)$, as shown on
the right picture in \eqref{mc2}. Furthermore, since
$$(k+1)\od_1-(k+2)c_1\leq(k+1)\bigl[(k+1)a_k+(k+2)a_{k+1}\bigr]
-(k+2)c_1=a_k,$$
we get $\oc_{k+1}\leq a_k$. This shows that
$(\omega_1,\omu_{k+1}-\mu_k)=k\oc_{k+1}-c_1+(k+1)a_{k+1}$ is
nonpositive, which implies that our candidate polytope is MV.
\end{itemize}
\medbreak

\noindent\textit{Proposition~\ref{pr:a10}:}

This is immediate from the explicit construction of $P_{\mathbf a}$
given above and from Lemmas~\ref{le:a_ka_k+1} and~\ref{le:a_ka^1}.
\medbreak

\noindent{\bf Inductive step where $a_k=0$ for all $k\geq2$:}
\medbreak

\noindent\textit{Theorem~\ref{th:Lusztig-data}:}

If $a^k=0$ for all $k\geq2$, then Theorem~\ref{th:Lusztig-data}
for $\mathbf a$ follows from Lemma~\ref{le:basecases} (these cases should really be thought of as the initialization of the induction).

Suppose now that there is $k\geq2$ such that $a^k>0$, and
consider the Lusztig data
$\mathbf a'=(a^1,a^2,\ldots,\lambda,\ldots,a_2,a_1)$.
Applying one of our induction steps above to $\mathbf a'$, we see that
Theorem~\ref{th:Lusztig-data} holds for $\mathbf a'$. (One notices
here that $\wt(\mathbf a')$ and $\wt(\mathbf a)$ have the same
height, so this trick does not drive us backwards in the induction.)
The image of $P_{\mathbf a'}$ by vertical reflection is then the
unique MV polytope $P_{\mathbf a}$.
\medbreak

\noindent\textit{Proposition~\ref{pr:a1not0}:}

Our assumption here is that $a_1>0$ and $a_k=0$ for all $k\geq2$, which
means that $\mu_1=\mu_\infty$. Let ${\bf b}$, ${\bf \ob}$ be as in the statement. There is a unique $P_{\bf b} \in \MV$ with right Lusztig data $\bf b$. If $\ob_1 \geq a_1$, then by Lemma \ref{le:a_1a^1} we can glue $P_{\bf b}$ and $P_{\mathbf d}$ for ${\bf d}=(a_1,0,\ldots,0,\ob_1)$ to get an MV polytope, and the statement follows from Lemma \ref{le:a_1a^1} \eqref{a_1a^1:Vertices}. Otherwise, the $\alpha_0$ diagonal $(\omu_\infty, \mu_\infty)$ cannot be active, and so instead the $\alpha_1$ diagonal $(\mu_\infty, \omu_\infty)$ must be active. Then all the $\oa_k$ are in fact $0$, and again the statement holds. 
\medbreak

\noindent\textit{Proposition~\ref{pr:a10}:}

We are here in the situation where $a_k=0$ for all $k\geq1$,
whence $\mu_0=\mu_\infty$. Then the $\alpha_0$-diagonal
$(\omu_\infty,\mu_\infty)$ is active in $P_{\mathbf a}$ and
$\omu_\infty=\omu_1$, which implies that $\oa_k=0$ for all
$k\geq2$. The result then follows from the observation that
$\mathbf d=\mathbf a$.

\medbreak

These three cases cover all possibilities, so the induction is complete.
\end{proof}

\begin{Proposition} \label{pr:a1ora1}
For any non-trivial $P\in\MV$, either $a_1$ or $\oa_1$ is nonzero,
and either $a^1$ or $\oa^1$ is non-zero.
\end{Proposition}
\begin{proof}
Let $\mathbf a$ be a Lusztig data such that $a_1=\oa_1=0$.
By Proposition~\ref{pr:a10}, we see that all $a_k$ are zero.
The same argument applied after horizontal reflection
gives that all $\oa_k$ are zero. In the MV polytope $P_{\mathbf a}$,
we thus have $\mu_0=\mu_\infty=\omu_\infty$. By Definition
\ref{def:characterization}~(\ref{part:middle2}), we get
$\lambda=\overline\lambda=0$. We thus have $\mu^\infty=\omu^\infty$.
Obvious weight considerations then show that all $a^k$ and all
$\oa^k$ are zero. Thus $\mathbf a$ is trivial.
\end{proof}

\begin{Remark} \label{rem:G=phi}
The induction used to prove Theorem~\ref{th:Lusztig-data} can be
viewed as an algorithm that explicitly constructs $P_{\mathbf a}$
and $\overline{\mathbf a}$ from $\mathbf a$. Following this
algorithm step by step, one can check that
\begin{align*}
\oa_1 = \max \bigl\{ & (k-1) a_k + (k-2) a_{k-1} - 2 a_{k-2} - \cdots - 2a_2 - 2 a_1, \text{ for } k \ge 2, \\
&\lambda_1 -  \cdots - 2a_3-2a_2-2a_1, \\
& k a^k + (k+1) a^{k+1}  + 2 a^{k+2} + 2 a^{k+3} + \cdots - 2 a_{k-2} - \cdots - 2a_2 - 2 a_1, \text{ for }k \ge 1 \bigr\}.
\end{align*}
\end{Remark}

\section{Crystal structure}

\begin{Definition} \label{def:crystal-ops}
Fix $P \in \MV$ with right and left Lusztig data ${\bf a}$ and ${\bf \oa}$ respectively. Then
$e_0(P)$ is the MV polytope with right Lusztig data $e_0({\bf a})$ and $e_1(P)$ is the MV polytope with left Lusztig datum $e_1({\bf \oa})$, where $e_0(\bf a)$ agrees with $\bf a$ except that $e_0({\bf a})^1= a^1+1$, and $e_1({\bf \oa})$ agrees with ${\bf \oa}$ except that $e_1({\bf \oa})^1 = \oa^1+1$. 

Similarly, $f_0(P)$ is the MV polytope with right data $f_0({\bf a})$ and $f_1(P)$ is the MV polytope with left Lusztig data $f_1({\bf \oa})$, where $f_0(\bf a)$ agrees with $\bf a$ except that $f_0({\bf a})^1= a^1-1$ and $f_1({\bf \oa})$ agrees with ${\bf \oa}$ except that $f_1({\bf \oa})^1 = \oa^1-1$, and if $a^1$ or $\oa^1$ are zero, then $f_0$ or $f_1$ sends that polytope to $\emptyset. $
\end{Definition}

\begin{Definition}
Fix $P \in \MV$ with right Lusztig data ${\bf a}$ and left Lusztig data ${\bf \oa}$. 
\begin{enumerate}
\item $\wt(P)= \mu^0-\mu_0$. 

\item $\varphi_0(P) = a^1$, $\varphi_1(P) = \oa^1$.

\item $\eps_0(P) = \varphi_0(P)- (\wt(P), \alpha_0)$ and $\eps_1(P) = \varphi_1(P)- (\wt(P), \alpha_1).$ 

\end{enumerate}
\end{Definition}

\begin{Definition}
The involution $*$ on $\MV$ negates the polytope. Recalling that elements of $\MV$ are only defined up to translation, this can be described algebraically as follows: for Lusztig data
$${\bf a}= (a_1, a_2, \ldots, {\bf \lambda}, \ldots, a^2, a^1),$$
$P^*_{\bf a}$ has left Lusztig data $\mathbf b=(\ob_k, \nu_k, \ob^k)$, where $\ob_k = a^k,\ \nu_k=\lambda_k,\ \ob^k = a_k.$

Define operators $e_0^*, e_1^*, f_0^*, f_1^*$ on $B(-\infty)$ by
$e^*_0 = * \circ e_0 \circ *, \;\; e^*_1 = * \circ e_1 \circ *, \;\; f^*_0 = * \circ f_0 \circ *, \;\; f^*_1 = * \circ f_1 \circ *,$ and similarly 
$\eps^*_0= \eps_0 \circ *, \;\; \eps^*_1=  \eps_1 \circ *, \;\; \varphi^*_0= \varphi_0 \circ *, \;\; \varphi^*_1= \varphi_1 \circ *.$
\end{Definition}

\begin{Remark}
The $*$ operators can also be defined as in Definition \ref{def:crystal-ops}, but where $\oa_1$ and $a_1$ get modified instead of $a^1$ and $\oa^1$. 
\end{Remark}

\begin{Theorem} \label{th:is-crystal}
$(\MV, e_0, e_1, f_0, f_1, \eps, \varphi, \wt)$ is isomorphic to the crystal $B(-\infty)$. Furthermore, 
$*$ is Kashiwara's involution. 
\end{Theorem}

The proof of Theorem \ref{th:is-crystal} will be delayed until Section \ref{sec:proofs2}.

\begin{Corollary} \label{cor:Bl}
Fix a dominant integral weight $\Lambda = c_0 \Lambda_0 + c_1 \Lambda_1$ for $\asl_2$, and let $\MV^\Lambda$ to be the subset of $\MV$ consisting of those MV polytopes $P_{\mathbf a}$ such that $a_1\leq c_1$ and $\oa_1\leq c_0$. Let $e_i^\Lambda, f_i^\Lambda$ be the operators inherited from the operators on $\MV$ by setting all $P \not \in \MV^\Lambda$ equal to $\emptyset$.
Let $\wt^\Lambda=\wt-\Lambda$, $\eps^\Lambda_i=\eps_i + c_i$ and $\varphi^\Lambda_i = \varphi_i$. 
 Then $(\MV^\Lambda, e_i^\Lambda, f_i^\Lambda, \varphi^\Lambda_i, \eps^\Lambda_i, \wt^\Lambda)$ is isomorphic to $B(-\Lambda)$. 
\end{Corollary}

\begin{proof}
This is immediate from Theorem \ref{th:sub-crystal} and Theorem \ref{th:is-crystal}.
\end{proof}

Let $ P $ be an MV polytope.  Recall that MV polytopes are only defined up to translation by weights.  For a dominant weight $ \Lambda $, let $ P^\Lambda $ denote the representative for $ P $ whose lowest vertex $ \mu_0 $ equals $ -\Lambda $.

\begin{Theorem} \label{th:contained}
For any dominant weight $ \Lambda $, we have $$ \MV^\Lambda = \{ P \in \MV : P^\Lambda \subset \mathrm{Conv} W (-\Lambda)\}, $$
where $ W $ denotes the Weyl group of $ \asl_2$ and $ \mathrm{Conv} $ denotes convex hull.  
\end{Theorem}

\begin{proof}
Write $ \Lambda = c_0 \Lambda_0 + c_1 \Lambda_1 $.

If $ P^\Lambda \subset \mathrm{Conv} W (-\Lambda) $, then the containment of the vertices $ \mu_1 $ and $ \omu_1 $ forces $ a_1 \le c_1 $ and $ \oa_1 \le c_0 $. 

We now turn to the converse implication: we pick $P \in \MV^\Lambda$,
and we show that $P^\Lambda $ is contained in the convex hull of
$W (- \Lambda)$.
Discarding a trivial case, we assume that $\Lambda\neq0$. We need the following, which is well known, but we will provide a proof for completeness. 

{\bf Claim:} The convex hull of $W(-\Lambda)$ is stable
by translation by $t\delta$ for any $t\geq0$.

{\bf Proof of claim:} 
Let $\alpha$ in the classical root lattice and let
$t_{\pm\alpha}\in W$ be the translation by $\pm\alpha$.
Any element $\lambda\in W(-\Lambda)$ has the same level as $-\Lambda$,
namely $-(c_0+c_1)$. Using \cite[(6.5.2)]{Kac:1990}, we see that
the midpoint of $t_\alpha\lambda$ and $t_{-\alpha}\lambda$ is
$\lambda+\frac{c_0+c_1}2|\alpha|^2\delta$. Choosing $\alpha$ large
enough, we deduce that $\lambda+t\delta$ belongs to the convex hull
of $W\lambda$ for any $t\geq0$. The claim follows. 

It is now clear that, to check that $ P^\Lambda \subset \mathrm{Conv} W (-\Lambda) $, it suffices to check that all vertices $\mu_k$ and $\omu_k$
of $ P^\Lambda $ belong to this convex hull. Set
$$v_k= s_1 s_0 s_1 \cdots s_i (-\Lambda) \quad \text{and} \quad \overline{v}_k= s_0 s_1 s_0 \cdots s_i (-\Lambda),$$
where in each case there are exactly $k$ reflections. It is actually enough to show that, for each~$k$,  
$$v_k-\mu_k \in \mathrm{span}_{\br_{\ge 0}} (\alpha_1, \alpha_1+(k-1) \delta ) \quad \text{and} \quad
\overline{v}_k-\omu_k \in \mathrm{span}_{\br_{\ge 0}} (\alpha_0, \alpha_0+(k-1) \delta).
$$
We do this by induction on $k$, as illustrated in Figure \ref{fig:convhull} and explained below

\begin{figure}[ht]
\begin{center}
\begin{tikzpicture}[xscale=0.8, yscale=0.2]
\coordinate (v0) at(0,0);
\coordinate (v1) at(1,1);
\coordinate (v2) at(2,4);
\coordinate (v3) at(3,9);
\coordinate (v4) at(4,16);
\coordinate (v5) at(5,25);
\coordinate (ov1) at(-1,1);
\coordinate (ov2) at(-2,4);
\coordinate (ov3) at(-3,9);
\coordinate (ov4) at(-4,16);
\coordinate (ov5) at(-5,25);
\coordinate (w5) at(3.5,20.5);
\coordinate (mu4) at(2.5,11.5);
\coordinate (mu5) at(3.1,16.9);
\coordinate (omu4) at(-3.2,13.8);

\draw (v0) node{$\bullet$};
\draw (v1) node{$\bullet$};
\draw (v2) node{$\bullet$};
\draw (v3) node{$\bullet$};
\draw (v4) node{$\bullet$};
\draw (v5) node{$\bullet$};
\draw (ov1) node{$\bullet$};
\draw (ov2) node{$\bullet$};
\draw (ov3) node{$\bullet$};
\draw (ov4) node{$\bullet$};
\draw (ov5) node{$\bullet$};
\draw (w5) node{$\bullet$};
\draw (mu4) node{$\bullet$};
\draw (mu5) node{$\bullet$};
\draw (omu4) node{$\bullet$};

\draw[line width = 0.5mm] (-5.3,28.3)--(ov5)--(ov4)--(ov3)--(ov2)--(ov1)
  --(v0)--(v1)--(v2)--(v3)--(v4)--(v5)--(5.3,28.3);
\draw[line width = 0.2mm] (v5)--(ov4);
\draw[line width = 0.2mm] (v4)--(1,13);
\draw[line width = 0.2mm] (v4)--(3,7);
\draw[line width = 0.2mm] (-1,13)--(ov4)--(-4.7,9.7);
\draw[line width = 0.2mm] (omu4)--(w5)--(mu4);

\draw (v0) node[below]{$v_0$};
\draw (v1) node[below]{$v_1$};
\draw (ov1) node[below]{$\overline v_1$};
\draw (v5) node[right]{$v_k$};
\draw (v4) node[right]{$v_{k-1}$};
\draw (ov4) node[left]{$\overline v_{k-1}$};
\draw (w5) node[above right]{$w_k$};
\draw (omu4) node[below right]{$\omu_{k-1}$};
\draw (mu4) node[left]{$\mu_{k-1}$};
\draw (mu5) node[left]{$\mu_k$};

\end{tikzpicture}
\end{center}
\caption{Diagram accompanying the proof of Theorem \ref{th:contained}.}
\label{fig:convhull}
\end{figure}
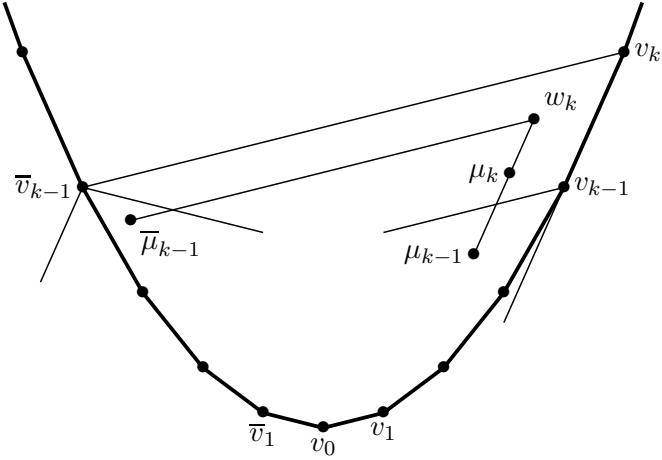

Let $w_k$ be the intersection of the line through $\omu_{k-1}$ parallel to $\alpha_1$ and the line through $\mu_{k-1}$ parallel to $\alpha_1+(k-1) \delta$. By the definition of MV polytope $w_k-\mu_k = t  (\alpha_1 + (k-1) \delta) $ for some $t \geq 0$ (where $t=0$ if and only if the diagonal $(\omu_{k-1}, \mu_k)$ is active). 

By induction, we can write 
$$v_{k-1}- \mu_{k-1} = x \alpha_1 + y (\alpha_1+ (k-1) \delta) \quad \text{and} \quad
\overline v_{k-1}- \omu_{k-1} = \overline x (-\alpha_1) + \overline y (\alpha_1+ (k-1) \delta)$$
for $x,y, \overline x,\overline y\geq 0$
(since the positive cone spanned by $\{  \alpha_1,  \alpha_1+ (k-1) \delta \}$ contains that spanned by $\{ \alpha_1, \alpha_1+ (k-2) \delta \}$, and the positive cone spanned by $\{-\alpha_1, \alpha_1+ (k-1) \delta\}$ contains that spanned by $\{ \alpha_0, \alpha_0+ (k-2) \delta \}$). 
It follows that
 $$v_k- w_k= x \alpha_1 + \overline y (\alpha_1 + (k-1) \delta),$$
 and hence 
 $$v_k- \mu_k= x \alpha_1 + \overline y (\alpha_1 + (k-1) \delta)+ t(\alpha_1 + (k-1) \delta).$$
Notice that $x$ and $\overline y + t$ are both positive. A similar argument shows that $\overline{v}_k-\omu_k \in \text{span}_{\br_{\ge 0}} (\alpha_0, \alpha_0+(k-1) \delta)$, completing the induction. 
\end{proof}

\subsection{Proof of Theorem \ref{th:is-crystal}}  \label{sec:proofs2}

Before beginning, let us give an outline of the proof. We first show that $\MV$ is a lowest weight combinatorial crystal. Hence 
by Corollary \ref{cor:comb-characterizaton2}, 
it remains to verify that,
for all $P \in \MV$, $i,j= 0$ or $1$, and $k \geq 0$,
\begin{align}
& \label{abc1} \text{ If $i \neq j$ then, $f_i^*f_j(P)= f_jf_i^*(P)$,} \\
&\label{abc2} \text{ If $\eps_i^*( (f_i)^{\varphi_i(P)}(P)) < \varphi_i( P)$ then $f_i(P) = f_i^*(P)$, and }\\
& \label{abc3} \text{ If $\eps_i^*( (f_i)^{\varphi_i(P)}(P)) \geq \varphi_i( P)$ then  $f_i^k f_i^*(P) = f_i^*f_i^k(P)$ for all $k \geq 0$.}
\end{align}
Condition \eqref{abc1} is straightforward. To handle \eqref{abc2} and \eqref{abc3} we first give a very explicit description of the operators $f_i$. This allows us to show that if $P \in \MV$ has an active $\alpha_1$ diagonal, then $f_0^* f_0^k(P) = f_0^k f_0^*(P)$, and otherwise $f_0(P)= f_0^*(P)$. It then remains to show that the condition $\eps_0^*( (f_0)^{\varphi_0(P)}(P)) \geq \varphi_0( P)$ is equivalent to the condition that $P$ has an active $\alpha_1$ diagonal, and similarly with the roles of $0$ and $1$ reversed.

\begin{Lemma} \label{lem:is-comb}
$(\MV, e_i, f_i, \varepsilon, \varphi, \wt)$ is a lowest weight combinatorial crystal.
\end{Lemma}

\begin{proof}
One readily checks that $(\MV, e_i, f_i, \varepsilon, \varphi, \wt)$ is a combinatorial crystal that enjoys condition (ii) in Definition~\ref{def:lwcrystal}. We prove condition (i) by induction on the weight. The lowest weight element is the trivial polytope (i.e., a single point, up to translation). Let $P$ be a non-trivial integral MV polytope. By Proposition \ref{pr:a1ora1}, either $a^1 \neq 0$ or $\oa^1 \neq 0$. Then by definition $f_i(P) \neq \emptyset$ for $i=0$ or $i=1$. Since $f_i(P)$ has smaller weight than $P$, by our induction hypothesis, the trivial polytope can be reached by applying a sequence of lowering operators $f_j$ to $f_i(P)$. Adding $f_i$ to this sequence, we get the desired property for $P$.
 \end{proof}

\begin{Remark}
\label{Rem:determbydiag}
Fix a Lusztig datum ${\bf a}$, a complete system of diagonals $S$, and a length $\ell(d)$ for each diagonal $d \in S$. It is clear that there can be at most one decorated GGMS polytope $P$ with right Lusztig data ${\bf a}$, and such that
\begin{itemize}
\item $P$ satisfies Definition \ref{def:characterization} parts \eqref{part:middle1} and \eqref{part:middle2}.

\item For each diagonal $d \in S$, the corresponding inequality from Definition \ref{def:characterization} part \eqref{top-part} or \eqref{bottom-part} holds with equality (so e.g. if $d= (\omu^k, \mu^{k-1}),$ then $(\omu^{k} -\mu^{k-1}, \omega_0)=0$), and has length $\ell(d)$.
\end{itemize}
In addition, for a fixed complete system of active diagonals $S$,
the restriction of the map $\mathbf a\to\overline{\mathbf a}$ to the set
$\{\mathbf a:\text{the elements of }S\text{ are active in }P_{\mathbf a}\}$
is linear.
\end{Remark}

The action of the lowering operator $f_0$ on $\MV$ can be concretely
described, as shown in the following proposition (see also Figure
\ref{fig:explicit-f}).

\begin{Proposition} \label{prop:explicit-e}
Let $\mathbf a$ be a Lusztig data. For $t\in[0,a^1]$, consider the
Lusztig data $\mathbf a_t=(a_1,a_2,\ldots,a^2,a^1-t)$, and denote by
$P_t$ the polytope with right Lusztig data $\mathbf a_t$.
These polytopes $P_t$ can be determined step by step, starting from $P_0$,
as follows.

Let $0\leq t<a^1$. Let $d$ be the highest active $\alpha_1$ diagonal in
$P_t$. Let $S$ be a complete system of active diagonals which contains
$d$ and whose elements are active in $P_t$. Then, for $\tau>t$ close enough
to $t$, $P_{\tau}$ can be determined by the following conditions:
\begin{itemize}
\item
The diagonals in $S$ are active in $P_{\tau}$. 
\item
Below (and including) $d$, the polytopes $P_t$ and $P_{\tau}$ are identical.
\item
Strictly above $d$, the length of a diagonal in $P_{\tau}$ has the form
$c+t-\tau$, where $c$ is the length of this diagonal in $P_t$.
\end{itemize}

For a fixed $t$, this recipe determines $P_\tau$ for all
$\tau>t$ up to a certain value $t'$. At $t'$, a new $\alpha_1$ diagonal
$d'$ has appeared, above $d$, ready to play the role of $d$
for values $\tau>t'$. At the end of the process, the uppermost $\alpha_1$
diagonal $(\omu^2,\mu^1)$ is active and the edge $(\omu^1,\mu^0)$
coincides with this diagonal.

\end{Proposition}

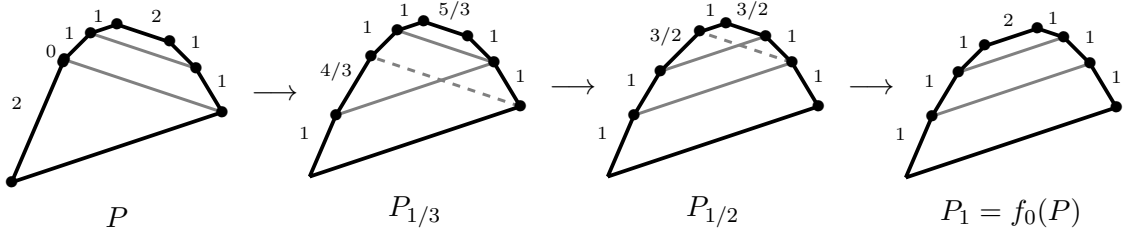
\begin{figure}[ht]

\begin{center}
\setlength{\unitlength}{0.5cm}
\begin{tikzpicture}[yscale=0.11666, xscale=0.35]

\draw [line width = 0.04cm, color=gray]
(4,-10)--(-2,-4)
(3,-5)--(-1,-1)
;

\draw 
(0,0) node {$\bullet$}
(2,-2) node {$\bullet$}
(3,-5) node {$\bullet$}
(4,-10) node {$\bullet$}
(-1,-1) node {$\bullet$}
(-2,-4) node {$\bullet$}
(-2.05,-4.3) node {$\bullet$}
(-4,-18) node {$\bullet$};

\draw [line width = 0.05cm] (0,0)--(2,-2)
(2,-2)--(3,-5)
(3,-5)--(4,-10)
(0,0)--(-1,-1)
(-1,-1)--(-2,-4)
(-2,-4)--(-4, -18)
(-4, -18)--(4,-10);

\draw
(1.5,1) node {\tiny $2$}
(3,-2) node {\tiny $1$}
(4,-6.5) node {\tiny $1$}

(-0.7,1) node {\tiny $1$}
(-1.8,-1) node {\tiny $1$}
(-2.5,-3) node {\tiny $0$}
(-3.8,-9) node {\tiny $2$}
;

\draw (6, -8) node {$\longrightarrow$};

\draw(0,-22) node {$P$};

\end{tikzpicture}
\hspace{-0.5cm}
\begin{tikzpicture}[yscale=0.11666, xscale=0.35]

\draw [line width = 0.04cm, color=gray]
(-3,-11)--(3,-5)
(3,-5)--(-0.6666,-1.3333)
;

\draw [line width = 0.04cm, color=gray, dashed]
(4, -10)--(-1.6666,-4.3333);

\draw 
(0.3333,-0.3333) node {$\bullet$}
(-0.6666,-1.33333) node {$\bullet$}
(-1.6666,-4.3333) node {$\bullet$}
(-3, -11) node {$\bullet$}
(2,-2) node {$\bullet$}
(3,-5) node {$\bullet$}
(4,-10) node {$\bullet$}
;

\draw [line width = 0.05cm] (0.3333,-0.3333)--(2,-2)
(2,-2)--(3,-5)
(3,-5)--(4,-10)
(0.3333,-0.3333)--(-0.6666,-1.33333)
(-0.6666,-1.3333)--(-1.6666,-4.3333)
(-1.6666,-4.3333)--(-3, -11)
(-3, -11)--(-4,-18)
(-4, -18)--(4,-10);

\draw
(1.5,1) node {\tiny $5/3$}
(3,-2) node {\tiny $1$}
(4,-6.5) node {\tiny $1$}

(-0.4,1) node {\tiny $1$}
(-1.8,-1) node {\tiny $1$}
(-3,-6) node {\tiny $4/3$}
(-4.2,-13) node {\tiny $1$}
;

\draw (6, -8) node {$\longrightarrow$};

\draw(0,-22) node {$P_{1/3}$};

\end{tikzpicture}
\hspace{-0.5cm}
\begin{tikzpicture}[yscale=0.11666, xscale=0.35]

\draw [line width = 0.04cm, color=gray]
(-3,-11)--(3,-5)
(-2,-6)--(2,-2)
;

\draw [line width = 0.04cm, color=gray, dashed]
(3, -5)--(-0.5,-1.5);

\draw 
(0.5,-0.5) node {$\bullet$}
(-0.5,-1.5) node {$\bullet$}
(-2,-6) node {$\bullet$}
(-3, -11) node {$\bullet$}
(2,-2) node {$\bullet$}
(3,-5) node {$\bullet$}
(4,-10) node {$\bullet$}
;

\draw [line width = 0.05cm] (0.5,-0.5)--(2,-2)
(2,-2)--(3,-5)
(3,-5)--(4,-10)
(0.5,-0.5)--(-0.5,-1.5)
(-0.5,-1.5)--(-2,-6)
(-2,-6)--(-3, -11)
(-3, -11)--(-4,-18)
(-4, -18)--(4,-10);

\draw
(1.5,1) node {\tiny $3/2$}
(3,-2) node {\tiny $1$}
(4,-6.5) node {\tiny $1$}

(-0.1,1) node {\tiny $1$}
(-1.8,-2) node {\tiny $3/2$}
(-3.1,-7) node {\tiny $1$}
(-4.2,-13) node {\tiny $1$}
;

\draw (6, -8) node {$\longrightarrow$};

\draw(0,-22) node {$P_{1/2}$};

\end{tikzpicture}
\hspace{-0.5cm}
\begin{tikzpicture}[yscale=0.11666, xscale=0.35]

\draw [line width = 0.04cm, color=gray]
(-3,-11)--(3,-5)
(-2,-6)--(2,-2)
;

\draw 
(1,-1) node {$\bullet$}
(-1,-3) node {$\bullet$}
(-2,-6) node {$\bullet$}
(-3, -11) node {$\bullet$}
(2,-2) node {$\bullet$}
(3,-5) node {$\bullet$}
(4,-10) node {$\bullet$}
;

\draw [line width = 0.05cm] (1,-1)--(2,-2)
(2,-2)--(3,-5)
(3,-5)--(4,-10)
(1,-1)--(-1,-3)
(-1,-3)--(-2,-6)
(-2,-6)--(-3, -11)
(-3, -11)--(-4,-18)
(-4, -18)--(4,-10);

\draw
(1.6,0.2) node {\tiny $1$}
(3,-2) node {\tiny $1$}
(4,-6.5) node {\tiny $1$}

(-0.1,0) node {\tiny $2$}
(-1.8,-2.5) node {\tiny $1$}
(-3.1,-7) node {\tiny $1$}
(-4.2,-13) node {\tiny $1$}
;

\draw (0,-22) node {$P_1=f_0(P)$};

\end{tikzpicture}

\end{center}

\caption{Applying $f_0$ to an MV polytope. By Proposition \ref{prop:explicit-e}, $f_0$ only affects the part of the polytope above the top active $\alpha_1$ diagonal, so that is all we draw. In each figure, the active diagonals used to continue the procedure are shown in grey. The dashed diagonals are also active, but are being ``replaced" by the solid diagonals that cross them.}
\label{fig:explicit-f}
\end{figure}

\begin{Remark}
There is a slight subtlety in this construction in the case when $d$ is the $\alpha_1$ diagonal $(\omu_\infty,\mu_\infty)$, since the procedure produces new $\alpha_1$ diagonals for all $\tau>0$ (all of which coincide with $(\omu^\infty,\mu^\infty)$,  which is now an $\alpha_1$ diagonal). The procedure actually proceeds smoothly until such time as a new $\alpha_1$ diagonal is formed which is not incident to $\mu^\infty$, which does not happen immediately. 
\end{Remark}

\begin{Remark}
Let $P\in\MV$, and let $\mathbf a$ be the right Lusztig data of $P$.
Then we have $P=P_0$, and Proposition \ref{prop:explicit-e} provides a concrete way to determine the polytope
$f_0^n(P)=P_n$, for all integers $0\leq n\leq a^1$.
One may also describe the action of $f_1$: the same algorithm can be run,
but twisted by a horizontal reflection.
One can run the algorithm in reverse to describe the action of the $e_i$.  In fact, we see that the action of $ e_0 $ slowly destroys active $ \alpha_1 $ diagonals and that once there are no more active $\alpha_1 $ diagonals, then $ e_0 $ acts by adding $ \alpha_0 $ to all the $ \omu_k, \omu^k $ vertices (except $ \omu_0 $).
\end{Remark}

\begin{Remark}
When the highest active $\alpha_1$ diagonal is incident to $\mu_\infty, \mu^\infty, \omu_\infty$ or $\omu^\infty,$ the above algorithm can affect $\olambda$.
\end{Remark}

\begin{proof}[Proof of Proposition \ref{prop:explicit-e}]
Let $0\leq t<a^1$. For $\tau>t$, let $Q_\tau$ be the GGMS polytope defined
by the three conditions listed in the proposition (we use
Remark \ref{Rem:determbydiag} to construct $Q_\tau$).
It is clear from continuity that 
when $\tau$ increases, a new active $\alpha_1$ diagonal will be created
just before $Q_\tau$ violates Definition \ref{def:characterization}
parts (\ref{top-part}) and (\ref{bottom-part}),
or before an edge length becomes negative.
We also note that the polytope $Q_\tau$ has the correct right Lusztig
data to be $P_\tau$, so as long as it is MV, it coincides with $P_\tau$.
So, the process leads to $P_{t'}$, from where we continue after switching
to a new active $\alpha_1$ diagonal.
\end{proof}

\begin{Proposition} \label{prop:ecom}
Fix $P \in \MV$. If $P$ has an active $\alpha_1$ diagonal, then $f_0^k f_0^* (P) = f_0^* f_0^k (P)$ for all $k$. If $P$ has no active $\alpha_1$ diagonal, then $f_0(P)= f_0^*(P)$. 
\end{Proposition}

\begin{proof}
Assume $P$ has an active $\alpha_1$ diagonal $d$, and let $P_\ell$ and $P_h$ be the two smaller polytopes obtained by cutting along $d$, where $P_\ell$ is below $P_h$. Then by Proposition \ref{prop:explicit-e}, $f_0$ and $f_0^*$ preserve the diagonal $d$, $f_0$ only affects $P_h$, and $f_0^*$ only affects $P_\ell$ (where the statements about $f_0^*$ follow by symmetry). In particular $f_0^k f_0^*(P) = f_0^*f_0^k(P)$ for all $k \geq 0$. 

If $P$ does not have an active $\alpha_1$ diagonal, the algorithm for calculating $f_0(P)$ given in Proposition \ref{prop:explicit-e} must proceed without creating a new active $\alpha_1$ diagonal at time $t<1$: in fact, denoting by $\omu_k(t)=\omu_k-t\alpha_0$ and $\omu^k(t)=\omu^k-t\alpha_0$ the vertices of $P_t$, the quantities $(\omu_k(t)-\mu_{k+1},\omega_0)$ and $(\omu^k(t)-\mu^{k-1},\omega_0)$ are of the form $c-t$ with $c$ a positive integer, so cannot vanish before $t$ reaches the value $1$. It follows that $f_0(P)= f_0^*(P)$.
\end{proof}

\begin{Remark} \label{rem:hn}
Proposition~\ref{prop:ecom} implies that, for all $P \in \MV$ and $k \geq 0$, $\varphi_0^*(f_0^kP)\leq\varphi_0^*(P).$
\end{Remark}

\begin{Lemma} \label{lem:ineq-active}
For all $P\in\MV$ we have $\varphi_0(P)+\varphi_0^*(P)-(\alpha_0,\wt(P))\geq0$,
with equality if and only if all $\alpha_0$-diagonals are active in $P$.
\end{Lemma}
\begin{proof}
Let $\mathbf a$ and $\overline{\mathbf a}$ the right and left
Lusztig data of $P$. We first notice that
\begin{align*}
(\alpha_0,\wt(P))
&=2(a^1+a^2+a^3+\cdots-a_1-a_2-a_3-\cdots)\\
&=2(\oa_1+\oa_2+\oa_3+\cdots-\oa^1-\oa^2-\oa^3-\cdots).
\end{align*}
It follows that
$$\varphi_0(P)+\varphi_0^*(P)-(\alpha_0,\wt(P))=
(a_1-\oa_2)+(a_2-\oa_3)+(a_3-\oa_4)+\cdots
+(\oa^1-a^2)+(\oa^2-a^3)+(\oa^3-a^4)+\cdots.$$
In addition, the inequalities coming from the $\alpha_0$ diagonals
(see Definition~\ref{def:characterization}~(\ref{top-part})
and~(\ref{bottom-part})) give
$$0\leq(\omega_1,\mu_k-\omu_{k+1})=(a_1-\oa_2)+2(a_2-\oa_3)
+3(a_3-\oa_4)+\cdots+k(a_k-\oa_{k+1}),$$
with equality if and only if $(\omu_{k+1},\mu_k)$ is active, and
$$0\leq(\omega_1,\mu^{k+1}-\omu^k)=(\oa^1-a_2)+2(\oa^2-a^3)
+3(\oa^3-a^4)+\cdots+k(a^{k+1}-\oa^k),$$
with equality if and only if $(\mu^{k+1},\omu^k)$ is active.
Thus
$$
\begin{aligned}
\varphi_0(P) + \varphi_0^*(P) - (\alpha_0, \wt(P)) & = \sum_{r=1}^\infty (a_r -
\oa_{r+1}) + \sum_{r=1}^\infty (\oa^r -a^{r+1})  \\
& =
 \sum_{k = 1}^\infty \frac{1}{k(k+1)} \left( (\omega_1, \mu_k - \omu_{k+1}) +
(\omega_1, \mu^{k+1} - \omu^k) \right) \ge 0 
\end{aligned}$$
with equality if and only if all diagonals
$(\omu_{k+1},\mu_k)$ and $(\mu^{k+1},\omu^k)$ are active.
\end{proof}

\begin{Lemma} \label{lem:all-active}
A polytope $P\in\MV$ has an active $\alpha_1$-diagonal if and
only if $\eps_0^*((f_0)^{\max}(P))-\varphi_0(P)\geq0$.
\end{Lemma}
\begin{proof}
Let $P\in\MV$. Set $n=\varphi_0(P)$ and $Q=f_0^n(P)$. We have
$$\eps_0^*((f_0)^{\max}(P))-\varphi_0(P)
=\varphi_0^*(Q)-(\alpha_0,\wt(Q))+\varphi_0(P)-2n
=\varphi_0^*(Q)+\varphi_0(P)-(\alpha_0,\wt(P)).$$

Suppose first that $P$ has an active $\alpha_1$-diagonal. Then by (the proof of)
Proposition~\ref{prop:ecom}
$\varphi_0^*(P)=\varphi_0^*(Q)$. Hence using Lemma \ref{lem:ineq-active}
$$\eps_0^*((f_0)^{\max}(P))-\varphi_0(P)
=\varphi_0^*(P)+\varphi_0(P)-(\alpha_0,\wt(P))\geq0$$
as desired.

Suppose now that all $\alpha_1$-diagonals are inactive in $P$.
Then $n=\varphi_0(P)>0$, for otherwise we would have $\mu^1=\mu^0$,
which would force the $\alpha^1$-diagonal $(\omu^2,\mu^1)$
to be active. It then follows from Proposition~\ref{prop:ecom} and Remark \ref{rem:hn} that $\varphi_0^*(Q)=\varphi_0^*(f_0^nP)\leq\varphi_0^*(f_0P)=
\varphi_0^*(f_0^*P)<\varphi_0^*(P)$, and therefore
$$\eps_0^*((f_0)^{\max}(P))-\varphi_0(P)
<\varphi_0^*(P)+\varphi_0(P)-(\alpha_0,\wt(P)).$$
Since all $\alpha_0$-diagonals are active in $P$ the right-hand
side is $0$. 
\end{proof}

To finish the proof of Theorem \ref{th:is-crystal}, we must verify \eqref{abc1}, \eqref{abc2} and \eqref{abc3}. 
For $i=0, j=1$, \eqref{abc1} follows because $f_0^*$ only affect $\oa_1$ (not the rest of $\bf \oa$) and $f_1$ only affect $\oa^1$. For $i=0$, \eqref{abc2} and \eqref{abc3} follow from Lemma \ref{lem:all-active} and Proposition \ref{prop:ecom}. The other cases follow by symmetry. \qed

\section{$A_2^{(2)}$ MV polytopes}
We now describe MV polytopes for the only other rank two affine root system. One could perhaps rework the proof as in the $\asl_2$ case and directly prove that the MV polytopes we describe below have the desired properties. However, we find it easier to use Kashiwara similarity of crystals (Theorem \ref{th:similarity}), which provides an embedding of $B^{A_2^{(2)}}(-\infty)$ into $B^{\asl_2}(-\infty)$, where we use superscripts to distinguish data related to different root systems. There is a corresponding vector space isomorphism from the real spans of the $\asl_2$ simple roots to the real span of the $A_2^{(2)}$ simple roots, and essentially our $A_2^{(2)}$ MV polytopes are the pullbacks of certain $\asl_2$ MV polytopes under this map. 

\subsection{The polytopes}
Recall that $A_2^{(2)}$ is the affine Kac-Moody algebra with symmetrized Cartan matrix
\begin{equation}
\tilde N=\left(
\begin{array}{rr}
2 & -4 \\
-4 & 8
\end{array}
\right).
\end{equation}
Let $\tilde \alpha_0$ and $\tilde \alpha_1$ be the simple roots, where $(\tilde \alpha_0, \tilde \alpha_0)=2$ and $(\tilde \alpha_1, \tilde \alpha_1)=8$. The set of positive roots of $A_2^{(2)}$ is $\Delta^+_{re}  \sqcup \Delta^+_{im}$, where, setting $\tilde \delta= 2 \tilde \alpha_0+\tilde \alpha_1$,
\begin{equation}
\Delta^+_{re} = \{ \tilde \alpha_0+ k\tilde \delta, \tilde \alpha_1+2 k\tilde \delta, \tilde \alpha_0+\tilde \alpha_1+k\tilde \delta, 2 \tilde \alpha_0+ (2k+1)\tilde \delta \mid k \geq 0 \} \quad \text{and} \quad \Delta_{im}^+= \{ k \delta \mid k \geq 1 \}.
\end{equation} 
Here $\Delta^+_{re}$ is the set of real positive roots and $\Delta^+_{im}$ is the imaginary positive roots. These can be drawn as
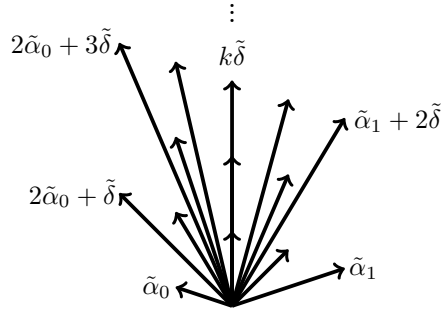
\begin{center}
\begin{tikzpicture}[scale=0.25]

\draw[line width = 0.05cm, ->] (0,0)--(-3,1);

\draw[line width = 0.05cm, ->] (0,0)--(6,2);
\draw[line width = 0.05cm, ->] (0,0)--(3,3);
\draw[line width = 0.05cm, ->] (0,0)--(0,4);
\draw[line width = 0.05cm, ->] (0,0)--(-3,5);
\draw[line width = 0.05cm, ->] (0,0)--(-6,6);

\draw[line width = 0.05cm, ->] (0,0)--(3,7);
\draw[line width = 0.05cm, ->] (0,0)--(0,8);
\draw[line width = 0.05cm, ->] (0,0)--(-3,9);

\draw[line width = 0.05cm, ->] (0,0)--(6,10);
\draw[line width = 0.05cm, ->] (0,0)--(3,11);
\draw[line width = 0.05cm, ->] (0,0)--(0,12);
\draw[line width = 0.05cm, ->] (0,0)--(-3,13);
\draw[line width = 0.05cm, ->] (0,0)--(-6,14);

\draw node at (0,16) {.};
\draw node at (0,15.6) {.};
\draw node at (0,15.2) {.};

\draw node at (-4,1) {{\small $\tilde \alpha_0$}};
\draw node at (-8.5,6) {{\small $2\tilde \alpha_0 + \tilde \delta$}};
\draw node at (-9.1,14) {{\small $2\tilde \alpha_0 + 3 \tilde \delta$}};

\draw node at (7,2) {{\small $\tilde \alpha_1$}};
\draw node at (8.8,10) {{\small $\tilde \alpha_1 + 2 \tilde \delta$}};

\draw node at (0,13.5) {{\small $k \tilde \delta$}};

\end{tikzpicture}
\end{center}

\begin{Definition} Label the real roots for $A_2^{(2)}$ by 
 $$
 r_k=\begin{cases}
 \tilde\alpha_1+(k-1)\tilde\delta&\text{if $k$ is odd,}\\
 \tilde\alpha_0+\tilde\alpha_1+\frac{k-2}2\tilde\delta&\text{if $k$ is even,}
 \end{cases}
 \qquad
 r^k=\begin{cases}
 \tilde\alpha_0+\frac{k-1}2\tilde\delta&\text{if $k$ is odd,}\\
 2\tilde\alpha_0+(k-1)\tilde\delta&\text{if $k$ is even.}
 \end{cases}
$$
\end{Definition}

\begin{Remark}
In the language of \cite{BKT:2011}, $r_1 \prec r_2 \prec \cdots \prec \delta \prec \cdots \prec r^2 \prec r^1$ and $r^1 \prec r^2 \prec \cdots \prec \delta \prec \cdots \prec r_2 \prec r_1$ are the two possible biconvex orders.
\end{Remark}

\begin{Definition}
An $A_2^{(2)}$ GGMS polytope is a convex polytope in $\text{span}_\br \{ \tilde \alpha_0, \tilde \alpha_1 \}$ such that all edges are parallel  to roots. Such a polytope is called integral if all vertices lie in $\text{span}_\bz \{  \tilde \alpha_0,  \tilde \alpha_1 \}$. 
\end{Definition}

As in the $\asl_2$ case, we can record a GGMS polytope by recording the positions of each vertex. The vertices are labeled $\mu_k, \mu^k, \overline \mu_k, \omu^k, \mu_\infty, \mu^\infty, \omu_\infty, \omu^\infty$, 
where e.g. $\mu_k$ records the vertex after the edge in the direction $r_k$ moving up the polytope on the right side, and $\omu^{k-1}$ records the vertex after the edge in the direction $r_k$ moving up the polytope on the left side. We also associate to a GGMS polytope $P$ the data $\{ a_k, a^k, \oa_k, \oa^k \}$ where e.g. $\mu_k-\mu_{k-1} = a_k r_k$. 

\begin{Definition} A decorated $A_2^{(2)}$ GGMS polytope is a GGMS polytope along with a choice of two sequences 
$\lambda = (\lambda_1 \geq  \lambda_2 \geq \cdots)$ and $\olambda = (\olambda_1 \geq \olambda_2 \geq \cdots)$ of positive real numbers such that $\mu^\infty - \mu_\infty = (\lambda_1 + \lambda_2+ \cdots) \tilde \delta$ and $\omu^\infty - \omu_\infty = (\olambda_1 + \olambda_2 + \cdots) \tilde \delta$ and, for all large enough $N$, $\lambda_N=\olambda_N=0$. 
Such a polytope is called integral if the underlying GGMS polytope is integral and all $\lambda_k, \olambda_k$ are integers.
\end{Definition}

\begin{Definition}
The right Lusztig data of a decorated $A_2^{(2)}$ GGMS polytope is the data ${\bf a} = ( a_k, \lambda_k, a^k )_{k \in \bn}$. The left Lusztig data is
 ${\bf \oa}=( \oa_k, \olambda_k, \oa^k )_{k \in \bn}$.
\end{Definition}

Let $\tilde \omega_0^\vee$ and $\tilde \omega_1^\vee$ be a choice of fundamental coweights (i.e. any choice of elements in weight space satisfying
$(\tilde \omega_i^\vee, \tilde\alpha_j) = \delta_{i,j}$, where we are using the bilinear form to identify weight space and coweight space). 

\begin{Definition}
An $A_2^{(2)}$ MV polytope is a decorated $A_2^{(2)}$ GGMS polytope such that
\begin{enumerate}

\item 
For each $k \geq 1$, $( \omu_{k}-\mu_{k-1}, \tilde \omega_1^\vee) \leq 0$ and $(\mu_{k} -\omu_{k-1}, \tilde \omega_0^\vee) \leq 0$, with at least one of these being an equality.

\item 
For each $k \geq 1$, $(\omu^{k} -\mu^{k-1}, \tilde \omega_0^\vee) \geq 0$ and $(\mu^{k} -\omu^{k-1},\tilde \omega_1^\vee) \geq 0$, with at least one of these being an equality.

\item 
If $(\mu_\infty, \omu_\infty)$ and $(\mu^\infty, \omu^\infty)$ are parallel then $\lambda=\olambda$. Otherwise, one is obtained from the other by removing a part of size $ (\mu_\infty - \omu_\infty, \tilde \alpha_1)/8$.

\item
$ \lambda_1, \olambda_1  \leq (\mu_\infty - \omu_\infty, \tilde \alpha_1)/8.$
\end{enumerate}
$P$ is called integral if it is integral as a decorated GGMS polytope. We denote by $\MV^{A_2^{(2)}}$ the set of integral $A_2^{(2)}$ MV polytopes up to translation. 

\end{Definition}

\subsection{Proof that they realize $B^{A_2^{(2)}}(-\infty)$}
Let $V= \text{span}_\br(\alpha_0, \alpha_1)$ and $\tilde V= \text{span}_\br(\tilde \alpha_0, \tilde \alpha_1)$. Consider the map $\gamma: V \rightarrow \tilde V$ defined by $\alpha_0 \mapsto \tilde \alpha_0$, $\alpha_1 \mapsto \tilde \alpha_1/2$. This sends each positive $\asl_2$ root to a multiple of a positive  $A_2^{(2)}$ root (the multiple is either $1$ or $1/2$), and is a bijection on root directions (i.e. roots up to scalar). It follows that $\gamma$ defines a bijection from real $\asl_2$ GGMS polytopes to real $A_2^{(2)}$ GGMS polytopes. Extend this map to a map of decorated GGMS polytopes by $\gamma(\lambda)_k= \lambda_k/2,$ $\gamma(\olambda)_k= \olambda_k/2$.

\begin{Lemma}  \label{lem:gbij}
$\gamma^{-1}(\MV^{A_2^{(2)}}) \subset \MV^{\asl_2}.$
\end{Lemma}

\begin{proof}
This follows from the observation that $\gamma^{-1}$ of a lattice point in $\text{span}_\br(\tilde \alpha_0, \tilde \alpha_1)$ is a lattice point in $\text{span}_\br(\alpha_0, \alpha_1)$. 
\end{proof}

\begin{Lemma} \label{lem:3C}
For $P \in \MV^{\asl_2}$, $\gamma(P) \in \MV^{A_2^{(2)}}$ if and only if any of the following three equivalent conditions  hold, where ${\bf a}$ and ${\bf \oa}$ are the right and left Lusztig data for $P$:
\begin{enumerate}
\item \label{ll1} $a_k, \oa^k$ are even whenever $k$ is odd, $a^k, \oa_k$ are even whenever $k$ is even, all $\lambda_k, \olambda_k$ are even. 

\item \label{ll2} $a_k$ is even whenever $k$ is odd, $a^k$ is even whenever $k$ is even, all $\lambda_k$ are even. 

\item \label{ll3} $\oa_k$ is even whenever $k$ is even, $\oa^k$ is even whenever $k$ is odd, all $\olambda_k$ are even. 
\end{enumerate}
\end{Lemma}

\begin{proof}
First we check that the three conditions are equivalent. Clearly \eqref{ll1} implies \eqref{ll2} and \eqref{ll3}. To see that \eqref{ll2} implies \eqref{ll1}, cut $P$ into quadrilaterals by Remark \ref{rem:cut-into-quads}. On each quadrilateral, the conditions on ${\bf a}$ from \eqref{ll2} implies the conditions from \eqref{ll1} on the data ${\bf \oa}$, and gluing these together \eqref{ll2} implies \eqref{ll1}. Showing \eqref{ll3} implies \eqref{ll1} is similar.

Condition \eqref{ll1} is exactly the condition that $\gamma$ sends each edge to a multiple of an $A_2^{(2)}$ root, or equivalently that $\gamma(P) \in \MV^{A_2^{(2)}}.$
\end{proof}

The symmetrized Cartan matrices $N$ for $\asl_2$ and $\tilde N$ for $A_2^{(2)}$ satisfy
$
\tilde N = D N D,
$
where $D= \text{diag} \{ 1, 2 \}$. Hence by
Theorem \ref{th:similarity} there is a unique embedding $S: B^{A_2^{(2)}}(-\infty) \rightarrow B^{\asl_2}(-\infty)$ such that $S(b^{N'}_-) = b^N_-$, and, for all $b \in  B^{A_2^{(2)}}(-\infty)$, $S(e_0 b) = e_0 S(b)$ and $S(e_1 b) = e_1^2 S(b)$. 

\begin{Lemma} \label{lem:Sset} 
Let $\MV^S = \{ P_{S(b)}: b \in B^{A_2^{(2)}}(-\infty) \}$. Then $\MV^S = \gamma^{-1} \MV^{A_2^{(2)}}$. 
\end{Lemma}

\begin{proof}
Using Lemma \ref{lem:3C}, one sees that $\gamma^{-1}(\MV^{A_2^{(2)}})$ is closed under the operators $e_0, f_0, e_1^2, f_1^2$. Furthermore, it is clear that, for any non-trivial $P \in \gamma^{-1}(\MV^{A_2^{(2)}})$, either $f_0(P) \neq \emptyset$ or $f_1^2(P) \neq \emptyset$. The claim now follows from the defining conditions on $S$ and the fact that $B^{A_2^{(2)}}(-\infty)$ is generated under the action of $e_0$ and $e_1$ by its lowest weight element. 
\end{proof}

\begin{Theorem}
There is a unique $\tilde P \in \MV^{A_2^{(2)}}$ for each choice of integral right Lusztig data.
Furthermore $\MV^{A_2^{(2)}}$ along with the crystal operators defined as in
Definition \ref{def:crystal-ops}, is a copy of $B^{A_2^{(2)}}(-\infty)$.
\end{Theorem}

\begin{proof}
For $b \in B^{A_2^{(2)}}(-\infty)$, let $\tilde P_b = \gamma (P_{S(b)})$. Lemma \ref{lem:Sset} shows that $b \rightarrow \tilde P_b$ is a bijection between $B^{A_2^{(2)}}(-\infty)$ and $\MV^{A_2^{(2)}}$. All the required properties now follow from the corresponding results about $\asl_2$ MV polytopes.
\end{proof}

\def\cprime{$'$} \def\cprime{$'$} \def\cprime{$'$} \def\cprime{$'$}
  \def\cprime{$'$} \def\cprime{$'$} \def\cprime{$'$}

\end{document}